\documentclass[hidelinks,onefignum,onetabnum]{siamart251216}

\usepackage{amsfonts}
\usepackage{graphicx}
\usepackage{epstopdf}
\usepackage{algorithmic}
\ifpdf
\DeclareGraphicsExtensions{.eps,.pdf,.png,.jpg}
\else
\DeclareGraphicsExtensions{.eps}
\fi

\newsiamremark{remark}{Remark}
\newsiamremark{hypothesis}{Hypothesis}
\crefname{hypothesis}{Hypothesis}{Hypotheses}
\newsiamthm{claim}{Claim}
\newsiamremark{fact}{Fact}
\crefname{fact}{Fact}{Facts}

\headers{An interior-point method with conjugate-free scaling}{Rui-Jin Zhang, Wenhao Fu, and Yu-Hong Dai}

\title{A primal--dual interior-point method for nonsymmetric conic optimization with conjugate-free scaling\thanks{Submitted to the editors DATE.
		\funding{ R.J.~Zhang was supported by the National Natural Science Foundation of China (No. 1250012017), the Natural Science Foundation of Tianjin, China (No. 24JCQNJC01970), and the Fundamental Research Funds for the Central Universities (Nos.  050-63253088 and 050-63263071). W.~Fu was supported by the Natural Science Foundation of Jiangsu Province (No. BK20250981). Y.H.~Dai was supported by the National Key  R\&D Program of China (Nos. 2021YFA1000300 and 2021YFA1000301) and the National Natural Science Foundation of China (No. 92473208).}}}

\author{Rui-Jin Zhang\thanks{School of Mathematical Sciences and LPMC, Nankai University, Tianjin, 300071, China 
		(\email{zhangrj@nankai.edu.cn}).}
	\and Wenhao Fu\thanks{School of Mathematical Sciences, Suzhou University of Science and Technology, Suzhou, 215009, China 
		(\email{wenhfu@usts.edu.cn}).}
	\and Yu-Hong Dai\thanks{State Key Laboratory of Mathematical Sciences, Academy of Mathematics and Systems Science, Chinese Academy of Sciences, Beijing 100190, China, and School of Mathematical Sciences,
		University of Chinese Academy of Sciences, Beijing 100049, China 
		(\email{dyh@lsec.cc.ac.cn}).}}
\usepackage{amsopn}
\usepackage{multirow}
\usepackage{amsmath,amssymb,amsfonts}
\usepackage{mathrsfs}
\usepackage[title]{appendix}
\usepackage{xcolor}
\usepackage{textcomp}
\usepackage{manyfoot}
\usepackage{booktabs}
\usepackage{algorithm,algorithmic}
\usepackage{graphicx}
\usepackage{subcaption}

\def\K{\mathbb{K}}

\def\dx{\Delta\bar{x}}
\def\dxx{\Delta {x}}
\def\dy{\Delta y}
\def\dz{\Delta z}
\def\ds{\Delta\bar{s}}
\def\dss{\Delta{s}}

\def\te{\omega^+}
\def\interior{\operatorname{int}}
\renewcommand{\[}{\begin{equation*}} 
	\renewcommand{\]}{\end{equation*}}

\newcommand{\intK}{\operatorname{int}(\mathbb{K})}
\newcommand{\norm}[1]{\| #1 \|} % 范数简写

\ifpdf
\hypersetup{
  pdftitle={A Primal--Dual Interior-Point Method for Nonsymmetric Conic Optimization with Conjugate-free Scaling},
  pdfauthor={Rui-Jin Zhang, Wenhao Fu, and Yu-Hong Dai}
}
\fi

\begin{document}

\maketitle

\begin{abstract}
We develop a primal--dual interior-point method for nonsymmetric conic optimization based on a conjugate-free scaling matrix. The scaling is obtained from a single-secant BFGS update of the primal barrier Hessian. In contrast to multi-secant BFGS scalings, it does not require conjugate-barrier derivatives. This feature is important for high-dimensional nonsymmetric cones, where conjugate-barrier derivatives may be unavailable in closed form or expensive to compute. We embed the conjugate-free scaling in a homogeneous self-dual predictor--corrector framework. Using a split central-path neighborhood that separately controls the conic variables and the scalar homogeneous variables, we prove that the scaling matrix remains uniformly comparable to the primal barrier Hessian. This comparison bound is used to prove neighborhood preservation and to show that the complementarity measure and the linear residual decrease at a uniform rate. Consequently, the method attains an iteration bound of $\mathcal{O}(\sqrt{\nu}\log(1/\varepsilon))$, improving the $\mathcal{O}(\nu\log(1/\varepsilon))$ bound of Badenbroek and Dahl [Optim. Methods Softw., 37 (2022), pp. 1027--1064] and matching the best-known complexity order for interior-point methods. Numerical experiments on instances involving the operator perspective epigraph cone and the quantum relative entropy cone show that the method is competitive with QICS, a specialized solver for conic models arising in quantum information.
\end{abstract}



\section{Introduction}\label{sec1}
Nonsymmetric conic optimization arises in many applications whose structure cannot be captured efficiently by nonnegative, second-order, or semidefinite cone formulations alone. In this paper, we consider the primal--dual conic pair
\begin{equation}\label{eq:coneopt}
\begin{array}{ll}
(\operatorname{P}) \qquad&\min\, \left\{\langle {c}, {x}\rangle\,:\, A {x}={b},\,x\in  \K\right\},\\
(\operatorname{D}) \qquad&\max \left\{\langle {b}, y \rangle\,:\, A^{\top} y+{s}={c},\,s\in  \K^*\right\},
\end{array}
\end{equation}
where $A \in \mathbb{R}^{m \times n}$ has full row rank, $\K$ is a proper cone that is nonsymmetric, and $\K^*$ is its dual cone.  Important examples of nonsymmetric cones include the exponential cone \cite{dahl2022primal,lindstrom2023error}, the power cone \cite{lin2024generalized,roy2022self}, the relative entropy cone  \cite{chandrasekaran2016relative}, and the quantum relative entropy cone \cite{fawzi2023optimal}. These cones arise in geometric programming \cite{skajaa2015homogeneous},  entropy and relative entropy optimization \cite{chandrasekaran2017relative},  and quantum information \cite{he2026operator,karimi2025efficient}. This range of applications has made nonsymmetric conic optimization an active research area in optimization theory and scientific computing.

Interior-point methods (IPMs) provide an efficient algorithmic framework for this class of problems. The self-concordant barrier theory of Nesterov and Nemirovskii \cite{nesterov1994interior} provides the fundamental polynomial-time framework for IPMs over general convex cones. Building on this theory, Nesterov \cite{nesterov2012towards} developed a systematic approach to nonsymmetric conic optimization inspired by self-scaled barriers. Skajaa and Ye \cite{skajaa2015homogeneous} subsequently proposed a homogeneous infeasible-start method that requires only the primal barrier and can also detect infeasibility. More recently, Papp and Varga \cite{papp2025interior} showed that primal--dual feasible methods using full Newton steps can attain the $\mathcal{O}\left(\sqrt\nu\log(1/\varepsilon)\right)$ iteration bound, while requiring only a logarithmically homogeneous self-concordant barrier for the primal cone. This matches the best-known complexity order for both symmetric and nonsymmetric conic optimization. 

%Alongside these theoretical advances, several solvers have been developed for nonsymmetric conic optimization, including Alfonso \cite{papp2022alfonso}, Clarabel \cite{chen2025efficient}, DDS \cite{karimi2024domain}, Hypatia \cite{coey2022solving,coey2022performance}, and QICS \cite{he2024qics}. These solvers work directly with the barriers of the original nonsymmetric cones. This avoids lifted reformulations involving symmetric cones,  which may introduce many auxiliary variables and constraints.

Alongside these theoretical developments, considerable progress has been made in the implementation of IPMs for nonsymmetric cones. Alfonso \cite{papp2022alfonso} provides an extensible implementation of the corrected Skajaa--Ye framework and enables optimization over a general cone once a suitable logarithmically homogeneous self-concordant barrier is available for the cone or its dual. Hypatia \cite{coey2022solving,coey2022performance} further demonstrated the computational benefits of working directly with a broad collection of nonsymmetric cones and developed several enhancements to its generic interior-point algorithm. From a complementary computational perspective, Chen and Goulart \cite{chen2025efficient} showed that the low-rank and sparse structure of barrier Hessians can be exploited to obtain efficient linear algebra for several important nonsymmetric cones. The same authors also developed Clarabel \cite{goulart2026clarabel}, a general-purpose interior-point solver that supports both symmetric and nonsymmetric cones and handles quadratic objectives directly. DDS \cite{karimi2024domain} takes a different approach based on the Domain-Driven framework. It implements an infeasible-start primal--dual method and supports a broad range of structured convex constraints. QICS \cite{he2024qics}, by contrast, is specialized for conic optimization problems arising in quantum information. Its strong reported performance on quantum relative entropy optimization makes QICS a particularly demanding benchmark for the numerical experiments in this paper. These solvers make it possible to work directly with the barriers of the original nonsymmetric cones and can avoid symmetric-cone reformulations that introduce additional variables and constraints.

A common strategy in primal--dual IPMs is to construct a scaling matrix that balances the primal and dual variables in a suitable local metric. For symmetric cones, the Nesterov--Todd (NT) scaling is a central construction. Nesterov and Todd \cite{nesterov1997self,nesterov1998primal} developed the theory of self-scaled barriers and showed that each interior primal--dual pair admits a unique scaling point whose barrier Hessian defines a positive definite scaling. This scaling satisfies two characteristic relations: it maps the primal variable to the dual variable and maps the conjugate-barrier gradient to the primal-barrier gradient. Moreover, the NT scaling can be bounded in terms of the primal and dual barrier Hessians, a property that plays a key role in the polynomial-time analysis. 
%Todd, Toh, and T\"ut\"unc\"u \cite{todd1998nesterov} provided a detailed computational study of the NT direction for semidefinite programming, showing that it can be computed efficiently and interpreted as a Newton direction. 
Tun\c{c}el \cite{tunccel1998primal} further related this structure to primal--dual symmetry and scale invariance. Subsequently, Hauser and G\"uler \cite{hauser2002self} established the close connection between self-scaled barriers and the symmetry groups of symmetric cones. Beyond these structural properties, the self-scaled framework has important algorithmic advantages. Nesterov and Todd \cite{nesterov1997self,nesterov1998primal} showed that it supports long-step primal--dual methods whose steps can typically move a large fraction of the way toward the boundary, rather than being restricted to the unit ball of the local barrier norm.

For nonsymmetric cones, however, the self-scaled structure underlying the NT scaling is generally unavailable, and a corresponding scaling point satisfying both primal--dual relations does not necessarily exist. A different approach is therefore to construct a positive definite scaling matrix directly from the current primal--dual pair while imposing analogues of the NT relations. Tun{\c c}el \cite{tunccel2001generalization} generalized primal--dual IPMs to arbitrary convex cones and showed that the two secant relations can be enforced through low-rank quasi-Newton updates of a general positive definite matrix, without relying on an NT scaling
point. Myklebust and Tun\c{c}el \cite{myklebust2014interior} further developed this framework and strengthened its connection with quasi-Newton updates. Dahl and Andersen \cite{dahl2022primal} specialized this approach to exponential-cone optimization and developed a practical method based on a multi-secant BFGS scaling. For the exponential cone, the resulting family of scalings is characterized by a single scalar, making the scaling particularly convenient to compute. Badenbroek and Dahl \cite{badenbroek2022algorithm} subsequently provided a rigorous polynomial-time analysis of this multi-secant BFGS scaling framework, thereby giving a theoretical foundation for the practical algorithm and establishing an $\mathcal{O}\left(\nu\log(1/\varepsilon)\right)$ iteration bound.

%established an $\mathcal{O}\left(\nu\log(1/\varepsilon)\right)$ iteration bound based on the multi-secant BFGS scaling.

Despite this progress, existing multi-secant BFGS scalings still require conjugate-barrier derivatives. In many nonsymmetric cones, evaluating the conjugate-barrier derivatives is substantially harder than evaluating the primal barrier derivatives. This issue was emphasized by Kapelevich, Andersen, and Vielma \cite{kapelevich2024computing}, who showed that conjugate-barrier information can be difficult to compute and may become a bottleneck in high-dimensional settings. This raises a natural question: 

\vspace{5pt}\noindent
\textit{Can one remove the secant equation involving conjugate-barrier derivatives and improve the $\mathcal{O}\left(\nu\log(1/\varepsilon)\right)$ iteration bound  established in \cite{badenbroek2022algorithm}?}
\vspace{5pt}

To address this question, we construct a conjugate-free scaling matrix. Starting from the primal barrier Hessian, we form a single-secant BFGS update that links the current primal and dual variables without requiring the secant equation involving conjugate-barrier derivatives. The resulting scaling matrix depends only on primal barrier information and the current primal--dual iterate. Hence it is easier to construct for general nonsymmetric cones, especially when conjugate-barrier derivatives are unavailable in closed form or expensive to compute. The main estimate shows that the resulting scaling matrix is uniformly comparable to the primal barrier Hessian in the split central-path neighborhood used by the algorithm. This comparison provides the estimates needed to establish an $\mathcal{O}(\sqrt{\nu}\log(1/\varepsilon))$ iteration bound.

We also test the proposed method on instances involving the operator perspective epigraph cone and the quantum relative entropy cone. These cones provide natural benchmarks because their barriers involve matrix logarithms, and the evaluation of derivatives can become a major computational cost in high dimensions.
%We compare the proposed method with QICS, which provides specialized implementations for conic models arising in quantum information. 
The numerical results show that the proposed method is competitive with QICS on the tested instances, while using a scaling construction that does not require conjugate-barrier derivatives.

The remainder of the paper is organized as follows. \Cref{sec2} recalls the barrier identities and local norm estimates used throughout the analysis.  \Cref{sec3} introduces the homogeneous self-dual embedding, the split central-path neighborhood, the initialization, the conjugate-free scaling matrix, and the predictor--corrector algorithm. \Cref{sec4} analyzes the predictor and corrector steps separately, proves neighborhood preservation for one predictor--corrector iteration, and establishes an $\mathcal{O}(\sqrt{\nu}\log(1/\varepsilon))$ iteration bound. \Cref{sec5} reports numerical experiments, and \Cref{sec6} concludes the paper.

\textbf{Notation.} Throughout the paper, let $\mathbb E$ be a finite-dimensional Euclidean space equipped with the inner product $\langle\cdot,\cdot\rangle$, and let $\|\cdot\|$ denote the Euclidean norm. For a proper cone $\K\subseteq\mathbb E$, we denote its interior by $\intK$. For a differentiable barrier $F$ on $\intK$, we denote by $g_x:=\nabla F(x)$ and $H_x:=\nabla^2F(x)$ the gradient and Hessian of $F$ at $x$. Whenever $H_x\succ0$, the associated local norm and dual local norm are
\[
\|u\|_x:=\langle u,H_xu\rangle^{1/2},\quad\|v\|_x^*:=\langle v,H_x^{-1}v\rangle^{1/2}.
\]
More generally, for any symmetric positive definite matrix $M$, we define
\[
\|u\|_M:=\langle u,Mu\rangle^{1/2},
\quad
\|v\|_M^*:=\langle v,M^{-1}v\rangle^{1/2}.
\]
The corresponding Cauchy--Schwarz inequality is
\[
|\langle u,v\rangle|=
|\langle M^{1/2}u,M^{-1/2}v\rangle|
\le
\|u\|_M\|v\|_M^*.
\]
We use $\|\cdot\|_2$ and $\|\cdot\|_F$ for the spectral norm and the Frobenius norm, respectively. For a linear operator $T:\mathbb E\to\mathbb E$,
we denote by $\|T\|_M^{\rm op}:=\|M^{1/2}TM^{-1/2}\|_2$ the operator norm induced by $\|\cdot\|_M$. For a symmetric matrix $B$,  we use the relative matrix norm $\|B\|_M:=\|M^{-1/2}BM^{-1/2}\|_2.$ When $M=H_x$, we use $\|\cdot\|_x^{\rm op}$ for the induced operator norm and $\|\cdot\|_x$ for the relative matrix norm.

\section{Preliminaries}\label{sec2}
This section recalls the basic estimates for self-concordant barriers and then specializes them to $\nu$-logarithmically homogeneous self-concordant barriers for proper cones. These results provide the norm estimates and barrier identities used in the scaling and complexity analysis.

\begin{definition}
\label{def:self-concordance}
Let $\mathbb{X}\subseteq \mathbb{E}$ be a nonempty open convex set. A convex function $F:\mathbb{X}\to\mathbb{R}$ is called self-concordant if it is three times continuously differentiable and~
\begin{equation}
    |D^3F(x)[h,h,h]|
    \le
    2\big(D^2F(x)[h,h]\big)^{3/2},
    \quad
 \text{for all } x\in\mathbb{X} \text{ and } h\in\mathbb E.
 \label{eq:self-concordance}
\end{equation}
In addition, if $F(x^{(k)})\to+\infty$ for every sequence $\{x^{(k)}\}\subseteq\mathbb{X}$ converging to a boundary point of $\mathbb{X}$, then $F$ is called a self-concordant barrier for $\mathbb{X}$.
\end{definition}

The following lemma collects the local estimates used in the analysis. The two bounds in \eqref{eq:aux_ineq} follow from \cite[Lemmas~4 and~5]{papp2017homogeneous}, while the Hessian perturbation bounds are standard consequences of self-concordant theory  \cite[Theorem~2.2.1]{renegar2001mathematical}.
\begin{lemma}
\label{lemma:perturb result}
Let $F$ be a self-concordant barrier on a nonempty open convex set $\mathbb{X}$ with $H_u\succ0$ for every $u\in\mathbb{X}$, and let $x\in\mathbb{X}$. If $u\in\mathbb E$ satisfies $\|u-x\|_x<1$, then $u\in\mathbb{X}$. For every $v\in\mathbb E$,
\begin{align}\label{eq:aux_ineq}
\left\| v \right\|_u^* \leq \frac{\left\|  v \right\|_x^*}{1 - \| u - x \|_x},\quad
\|g_u-g_x\|_{x}^{*}\le \frac{\|u-x\|_{x}}{1-\|u-x\|_{x}}.
\end{align}
Furthermore, the Hessian satisfies
\begin{equation}
    \big\|H_{u}^{-1}H_x\big\|_x^{\rm op}
    \le
    \frac{1}{(1-\|u-x\|_x)^2},
    \quad
    \big\|H_x^{-1}H_{u}-I\big\|_x^{\rm op}
    \le
    \frac{1}{(1-\|u-x\|_x)^2}-1 .
   \label{eq:sc_hessian_comp}
\end{equation}
\end{lemma}

We also need the following damped Newton residual bound for the case where the Hessian is replaced by an approximate matrix.
\begin{theorem}\label{thm:inexact_newton}
Let $F$ be a self-concordant barrier on a nonempty open convex set $\mathbb{X}$ with $H_u\succ0$ for every $u\in\mathbb{X}$, and let $x\in\mathbb{X}$. Suppose that $W\succ0$ satisfies
\begin{equation}\label{eq:relative_H_approx}
(1-\delta)\,H_x\ \preceq\ W\ \preceq\ (1+\delta)\,H_x
\quad\text{for some }\delta\in[0,1).
\end{equation}
Define the inexact Newton direction $\tilde n(x):=-W^{-1}g_x$, and for a step length $\alpha\in[0,1]$ set
$
x^+:=x+\alpha \tilde n(x),
\,
\rho:=\|x^+-x\|_x=\alpha\|\tilde n(x)\|_x.
$
Assume $\rho<1$. Then $x^+\in \mathbb{X}$ and the exact Newton decrement at $x^+$ satisfies
\begin{equation}\label{eq:main_bound}
\|n(x^+)\|_{x^+}
\ \le\
\left(\frac{\rho}{1-\rho}\right)^2
\ +\
\frac{(1-\alpha+\delta)\,\|\tilde n(x)\|_x}{1-\rho},
\end{equation}
where $n(x^+):=-H_{x^+}^{-1}g_{x^+}$.
\end{theorem}
\begin{proof}
The proof is given in Appendix~\ref{secA1}.
\end{proof}

When $\delta=0$, Theorem~\ref{thm:inexact_newton} reduces to the standard damped Newton estimate with the exact Hessian \cite[Theorem 3]{papp2017homogeneous}. The additional parameter $\delta$ quantifies the error caused by replacing $H_x$ with an approximate matrix. We now specialize to conic barriers with logarithmic homogeneity.

\begin{definition}
\label{def:lhscb}
Let $\K\subseteq \mathbb{E}$ be a proper cone with nonempty interior. A self-concordant barrier $F:\interior(\K)\to\mathbb{R}$ is called $\nu$-logarithmically homogeneous if there exists a constant $\nu>0$ such that
\begin{equation}\label{eq:log-homogeneity}
    F(tx)=F(x)-\nu\log t,
    \quad x\in\interior(\K),\ t>0.
\end{equation}
In this case, $F$ is called a $\nu$-logarithmically homogeneous self-concordant barrier, or $\nu$-LHSCB, for $\K$.
\end{definition}

Logarithmic homogeneity yields several identities that will be used repeatedly in the complexity analysis of the interior-point method. These identities are standard for $\nu$-LHSCBs; see, for example, \cite{dahl2022primal,nesterov1997self,nesterov1998primal}.
\begin{lemma}\label{lem:loghom-id}
Let $F$ be a $\nu$-LHSCB for $\K$.  Then, for every
$x\in\intK$,
	\begin{equation}\label{eq:Hxgrad}
		H_{{x}} x=-g_{ x},
		\quad
		\langle g_{ x}, x\rangle=-{\nu},
		\quad
		\|g_{ x}\|_{ x}^*=\sqrt{{\nu}}.
	\end{equation}
\end{lemma}

Throughout the paper, we use the fact that $\nu\ge1$ for any $\nu$-LHSCB on a nontrivial proper cone \cite[Corollary~2.3.3]{nesterov1994interior}.
Although the scaling matrix introduced later does not require explicit evaluation of conjugate-barrier quantities, the conjugate barrier is useful for describing the standard primal--dual correspondence of $\nu$-LHSCBs and for proving strict dual feasibility. The conjugate barrier is defined by
\begin{equation}
    F^*(s)
    :=
    \sup_{x\in\intK}
    \{-\langle s,x\rangle-F(x)\},
    \quad
    s\in\operatorname{int}(\K^*).
    \label{eq:conjugate-barrier}
\end{equation}
For $s\in\operatorname{int}(\K^*)$, we denote its gradient and Hessian by $g_s^*:=\nabla F^*(s)$ and $H_s^*:=\nabla^2F^*(s)$, respectively. The corresponding local norms are defined by
\begin{equation}
    \|u\|_s
    :=
    \langle u,H^*_s u\rangle^{1/2},
    \quad
    \|v\|_s^*
    :=
    \langle v,(H^*_s)^{-1}v\rangle^{1/2}.
    \label{eq:dual-barrier-local-norm}
\end{equation}

For a $\nu$-LHSCB, the
gradient mapping $\nabla F$ gives a bijection between
$\intK$ and $-\operatorname{int}(\K^*)$, and the Hessian of the conjugate barrier is the inverse of the primal Hessian at the corresponding point. We record this standard result below.

\begin{theorem}[\protect{\cite[Theorems 7 and 8]{papp2017homogeneous}}]\label{thm:dual interior}
Let $F$ be a $\nu$-LHSCB for $\K$. Then $F^*$ is a $\nu$-LHSCB
for $\K^*$, and the following properties hold:
\begin{enumerate}
\item[(i)] The gradient mapping $\nabla F:\intK \to -\operatorname{int}(\K^*)$ is a bijection.
\item[(ii)] If $x\in\intK$ and $s\in\operatorname{int}(\K^*)$ satisfy $s=-g_x$, then
\begin{equation*}
-g^*_s=x
\quad\text{and}\quad
H^*_s=H_x^{-1}.
\end{equation*}
Furthermore, we have $\|v\|_{-g_x}=\|v\|_x^*$ for all $v\in \mathbb{E}.$
\end{enumerate}
\end{theorem}

\section{Interior-point method for nonsymmetric conic optimization}\label{sec3}
This section presents the proposed primal--dual interior-point method for nonsymmetric conic optimization. We first introduce the homogeneous self-dual (HSD) embedding and the split central-path neighborhood, then specify an initial point, construct the conjugate-free scaling matrix, and finally state the predictor--corrector algorithm.
\subsection{Homogeneous self-dual embedding and split central-path neighborhood}
Conic optimization problems may be infeasible or unbounded. To detect such cases effectively, the HSD embedding has been widely used.

Let $F$ be a $\nu$-LHSCB for $\K$. The associated barrier subproblem for $(\operatorname{P})$ in \eqref{eq:coneopt} takes the form
\begin{equation}\label{eq:barrier subproblem}
\min \left\{ \langle c,x\rangle+\mu F(x) \::\: Ax = b,\,x\in \operatorname{int}(\K)\right\},
\end{equation}
where $\mu >0$ denotes the barrier parameter. By gradually decreasing $\mu$, we obtain an approximate optimal solution of $(\operatorname{P})$.

The Karush–Kuhn–Tucker (KKT) conditions for \eqref{eq:barrier subproblem} are
\begin{equation}\label{eq:kkt condition}
    \begin{aligned}
    Ax &= b,\,x\in \intK,\\
    A^{\top} y +s &=c,\,s\in \interior(\K^*),\\
    \mu g_x+s&=0,
    \end{aligned}
\end{equation}
with dual variables $(y,s)$. The third equation in \eqref{eq:kkt condition} is the central-path relation associated with the barrier subproblem.

We then introduce the HSD embedding by adding $\tau,\,\kappa\geq 0$ and defining
\begin{equation*}\label{eq:extended_vars}
\bar{x} := (x; \tau) \in \bar{\mathbb{K}} := \mathbb{K} \times \mathbb{R}_+, \quad \bar{s} := (s; \kappa) \in \bar{\mathbb{K}}^* := \mathbb{K}^* \times \mathbb{R}_+,
\end{equation*}
together with the extended barrier
\begin{equation*}\label{eq:extended_barrier}
\bar{F}(\bar{x}) := F(x) - \log \tau, \quad \bar{\nu} := \nu + 1.
\end{equation*}
If $F$ is a $\nu$-LHSCB for $\mathbb{K}$, then $\bar{F}$ is a $\bar{\nu}$-LHSCB for $\bar{\mathbb{K}}$. Group the variables as $z := (\bar{x}; y; \bar{s}) \in \mathcal{F} := \bar{\mathbb{K}} \times \mathbb{R}^m \times \bar{\mathbb{K}}^*.$
With this notation, the linear equations of the HSD embedding can be written compactly as
\begin{equation}\label{eq:HSD_linear}
G(y; \bar{x}) - (0; \bar{s}) = 0, \quad G := \begin{pmatrix} 0 & A & -b \\ -A^\top & 0 & c \\ b^\top & -c^\top & 0 \end{pmatrix}.
\end{equation}

The skew-symmetry of $G$ will be used repeatedly to derive complementarity identities for the predictor and corrector directions. Furthermore, any solution of the homogeneous system satisfies  $\langle x,s\rangle+\tau\kappa=0.$ If $\tau>0$, then $x/\tau$ recovers a primal optimal solution of \eqref{eq:coneopt}, while $(y/\tau,s/\tau)$ recovers a dual optimal solution. If $\tau=0$ and $\kappa>0$, then $b^\top y>0$ gives a certificate of
primal infeasibility, while $c^\top x<0$ gives a certificate of dual
infeasibility~\cite[Theorems 9.5 and 9.6]{wright1997primal}.

The complementarity residual $\psi( x, s,\mu)$, the cone complementarity measure $\mu$, and the homogeneous complementarity measure $\bar \mu$ are defined as
\[
\begin{aligned}
\psi( x, s,\mu):= s+\mu g_{{x}},
\quad
\mu:=\frac{\langle x, s\rangle}{\nu},
\quad
\bar{\mu}:=\frac{\langle\bar{x}, \bar{s}\rangle}{\bar{\nu}}.
\end{aligned}
\]
Here,  $\psi(x, s, \mu)$ measures the deviation from the central path, $\mu$ is the average complementarity for the conic variables, and $\bar{\mu}$ is the corresponding average complementarity for the homogeneous variables.

 We define the following neighborhood of the central path:
\begin{equation}\label{eq:neighborhood}
\mathcal N(\eta,\beta_l,\beta_u):=\Bigl\{z\in\interior(\mathcal{F}):\ \|\psi( x, s,\mu)\|_{ x}^{*}\le \eta\,\mu,\,\beta_l {\mu}\leq\kappa\tau\leq\beta_u {\mu}\Bigr\},
\end{equation}
where $\eta\in(0,1),\,0<\beta_l\leq\beta_u$. We call $\mathcal N(\eta,\beta_l,\beta_u)$ a split central-path neighborhood for the HSD embedding. The proximity condition $\|\psi( x, s,\mu)\|_{ x}^{*}\le \eta\,\mu$ keeps the conic variables close to the central path, while the bounds $\beta_l {\mu}\leq\kappa\tau\leq\beta_u {\mu}$ control the scalar homogeneous complementarity and keep it comparable to the cone complementarity. This split form is adopted because the scaling matrix constructed in this paper acts only on the conic variables $x$ and $s$. It is therefore naturally associated with the cone complementarity measure $\mu$, rather than with the homogeneous complementarity measure $\bar\mu$. Nevertheless, the two complementarity measures are uniformly equivalent throughout $\mathcal N(\eta,\beta_l,\beta_u)$. Indeed, since
\[
    \bar\mu
    =
    \frac{\langle \bar x,\bar s\rangle}{\bar\nu}
    =
    \frac{\langle x,s\rangle+\tau\kappa}{\nu+1}
    =
    \frac{\nu\mu+\tau\kappa}{\nu+1},
\]
and since $\beta_l\mu\le \tau\kappa\le \beta_u\mu$, we have
\[
    \frac{\nu+\beta_l}{\nu+1}\,\mu
    \le
    \bar\mu
    \le
    \frac{\nu+\beta_u}{\nu+1}\,\mu \quad\text{and}\quad  \frac{\nu+1}{\nu+\beta_u}\,\bar\mu
    \le
    \mu
    \le
    \frac{\nu+1}{\nu+\beta_l}\,\bar\mu .
\]
Hence, $\mu$ and $\bar\mu$ are bounded above and below by positive constant multiples of each other, where the constants depend only on $\beta_l$ and $\beta_u$. Consequently, the two measures are of the same asymptotic order, and driving either one to zero also drives the other to zero.

\subsection{Initialization}
\label{subsec:initialization}

We next specify an initial point in $\mathcal N(\eta,\beta_l,\beta_u)$. Assume that an interior point $x^{(0)}\in\intK$ is available, and that the barrier $F$, its gradient, and its Hessian can be evaluated at this point. For the cone families considered in \Cref{sec5}, such a point can be obtained explicitly. 

Choose any $y^{(0)}\in\mathbb R^m$, for example, $y^{(0)}=0$, and choose a scalar $\beta_0\in[\beta_l,\beta_u].$ Define
\begin{equation}
    s^{(0)}=-g_{x^{(0)}},\quad
    \tau^{(0)}=1,\quad
    \kappa^{(0)}=\beta_0,
\end{equation}
and set
\begin{equation}
    \bar x^{(0)}=(x^{(0)};\tau^{(0)}),\quad
    \bar s^{(0)}=(s^{(0)};\kappa^{(0)}),\quad
    z^{(0)}=(\bar x^{(0)};y^{(0)};\bar s^{(0)}).
\end{equation}
Then $s^{(0)}\in\operatorname{int}(\K^*)$ by Theorem~\ref{thm:dual interior}(i), and hence $z^{(0)}\in\operatorname{int}(\mathcal F)$. In addition, by Lemma~\ref{lem:loghom-id},
\[
    \mu^{(0)}
    =
    \frac{\langle x^{(0)},s^{(0)}\rangle}{\nu}
    =
    \frac{\langle x^{(0)},-g_{x^{(0)}}\rangle}{\nu}
    =
    1.
\]
Therefore,
$
    \psi(x^{(0)},s^{(0)},\mu^{(0)})
    =
    s^{(0)}+\mu^{(0)} g_{x^{(0)}}
    =
    0
$
and
$
    \tau^{(0)}\kappa^{(0)}
    =
    \beta_0
    =
    \beta_0\mu^{(0)}.
$
Since $\beta_0\in[\beta_l,\beta_u]$, we obtain
$
    \beta_l\mu^{(0)}
    \le
    \tau^{(0)}\kappa^{(0)}
    \le
    \beta_u\mu^{(0)}.
$
Thus,
$
    z^{(0)}\in\mathcal N(\eta,\beta_l,\beta_u).
$

The initial point need not satisfy the linear equations of the HSD embedding. Its linear residual
$
    r(z^{(0)}):=G(y^{(0)};\bar x^{(0)})-(0;\bar s^{(0)})
$
may be nonzero, and this infeasibility is reduced by the predictor step. Indeed, Lemma~\ref{lem:feas_reduction} shows that the linear residual is contracted by a factor at each predictor step.
\subsection{Conjugate-free scaling}
We construct the scaling matrix used in the Newton systems for nonsymmetric conic optimization. The construction avoids conjugate-barrier derivatives, which may be unavailable in closed form or expensive to compute.

In the symmetric-cone setting, suppose that $F$ is a self-scaled barrier for $\K$. For every primal--dual pair $(x,s)\in\operatorname{int}(\K)\times\operatorname{int}(\K^*)$, there exists a unique Nesterov--Todd (NT) scaling point $w\in\operatorname{int}(\K)$ \cite{nesterov1997self}. This scaling point satisfies 
\begin{equation}\label{NT algebraic conditions}
    H_w x=s,\quad H_w g_s^*=g_x.
\end{equation}
Thus $H_w$ maps the primal variable $x$ to the dual variable $s$, and maps the conjugate-barrier gradient at $s$ to the primal barrier gradient at $x$. Since $F$ is logarithmically homogeneous, its Hessian satisfies $H_{\alpha w}=\frac{1}{\alpha^2} H_w$ for $\alpha>0$. Therefore, with the normalized scaling point $\widehat w:=\sqrt{\mu}\,w$, we have $H_{\widehat w}=\frac{1}{\mu} H_w$. Consequently, \eqref{NT algebraic conditions} can be equivalently written as
\begin{equation}\label{NT normalized algebraic conditions}
    \mu H_{\widehat w}x=s,\quad
    \mu H_{\widehat w}g_s^*=g_x .
\end{equation}
For a nonsymmetric cone, a self-scaled barrier and the corresponding NT scaling point are generally unavailable. Nevertheless, one may still construct a positive definite matrix $W$ by requiring it to satisfy the following two secant equations
\begin{equation}\label{secant conditions}
\mu Wx=s,\quad \mu W g^*_s=g_x.
\end{equation}
The first equation maps the primal variable $x$ to the dual variable
$s$, whereas the second equation involves the conjugate-barrier derivative.  In contrast to the NT scaling on symmetric cones, these equations do not determine $W$ through a cone automorphism or a unique scaling point. Existing primal--dual scaling constructions seek a positive definite matrix that satisfies these equations and remains sufficiently close to the primal barrier Hessian for the Newton analysis.

The exponential cone provides an important example of this approach. For this cone, Dahl and Andersen \cite{dahl2022primal} use $H_x$ as the reference matrix and compute the scaling matrix $W$ by solving a least-change problem
\[
\begin{aligned}
\min\, \left\{\| \Omega^{1/2} (W^{-1} - H_x^{-1}) \Omega^{1/2} \|_F\, :\, W^{-1} s = \mu x,\,W^{-1} g_x  = \mu g^*_s,\,W \succ 0\right\},
\end{aligned}
\]
where $\Omega \succ 0$ is chosen so that the two secant equations in \eqref{secant conditions}  can be satisfied simultaneously. The corresponding update is the multi-secant BFGS scaling, whose explicit form can be written as
\[
\begin{aligned}
W
&=H_x + \frac{s s^\top}{\nu\mu^2} + \frac{\big( s + \mu g_x \big) \big( s + \mu g_x \big)^\top}{\mu\langle x + \mu g^*_s, s + \mu g_x \rangle} - \frac{g_x  g_x^\top}{\nu} - \frac{H_x \varrho (H_x \varrho)^\top}{\langle \varrho, H_x \varrho \rangle},
\end{aligned}
\]
with $\varrho = -g^*_s - \frac{\langle g_x, g^*_s \rangle}{\nu} x.
$

For the three-dimensional exponential cone, this construction is especially effective, since the corresponding scalings have only one remaining degree of freedom and an explicit matrix $W$ can be computed efficiently.  For general high-dimensional nonsymmetric cones, the evaluation of $g_s^*$ itself may become a substantial computational bottleneck \cite{kapelevich2024computing}, and therefore the above approach is difficult to extend beyond a few special cases. Motivated by this observation, we retain only the first secant equation $\mu W x=s$ and remove the second one. Specifically, starting again from $H_x$, we obtain $W$ by solving the following simplified least-change problem
\[
\begin{aligned}
\min\, \left\{\| \Omega^{1/2} (W^{-1} - H_x^{-1}) \Omega^{1/2} \|_F\,:\,W^{-1} s = \mu x,\,W \succ 0\right\},
\end{aligned}
\]
where $\Omega \succ 0$ satisfies $\mu\Omega x =s$. The associated single-secant BFGS update for $H_x$ is then
\[
W=H_x + \frac{s s^\top}{\nu\mu^2}- \frac{g_x  g_x^\top}{\nu}.
\]
The resulting matrix satisfies the secant equation $\mu W x=s$ without requiring the evaluation of $g_s^*$. This construction is therefore useful for high-dimensional nonsymmetric cones, where conjugate-barrier evaluations may be costly.

To analyze the scaled Newton system, we use the following auxiliary estimate in a relative matrix norm.
\begin{lemma}[\protect{\cite[Lemma 7.1]{myklebust2014interior}}]\label{lem:matrix_norm_bound}
Let $M\in \mathbb{S}^n_{++}$, $u$ and $v \in \mathbb{R}^n$. Then 
\[
\|uu^\top-vv^\top\|_M\leq \|u+v\|_M^*\|u-v\|_M^*.
\]
\end{lemma}

We use this bound to estimate the relative norm of the rank-two update in the scaling matrix.
\begin{theorem}\label{Thm:Hessian bound}
For $z\in \mathcal{N}(\eta,\beta_l,\beta_u)$, let $W = H_x + \frac{ss^\top}{\nu\mu^2}-\frac{g_x g_x^\top}{\nu}$. If $0<\eta<\sqrt2-1$, then
\begin{equation}\label{eq:scaling_matrix_bound}
    (1-\delta)\,H_{{x}}\ \preceq\ W\ \preceq\ (1+\delta)\,H_{{x}},
\end{equation}
where $\delta:=\eta^2+2\eta<1$.
\end{theorem}
\begin{proof}
Using Lemma \ref{lem:matrix_norm_bound} with $M=H_x$ and the proximity condition $\norm{s+\mu g_x}_x^*=\norm{\psi(x,s,\mu)}_x^*\le \eta\mu$, together with the barrier identity $\norm{g_x}_x^*=\sqrt{\nu}$ for a $\nu$-LHSCB, we obtain
\begin{equation*}\label{eq:scaling_proof}
\begin{aligned}
\norm{W - H_x}_x &= \left\|\frac{s s^\top}{\nu \mu^2} - \frac{g_x g_x^\top}{\nu}\right\|_x 
\leq \frac{\norm{s + \mu g_x}_x^* \norm{s - \mu g_x}_x^*}{\nu \mu^2} \\
&\leq \frac{\eta(\norm{\psi}_x^* + 2\mu \norm{g_x}_x^*)}{\nu \mu} \leq \frac{\eta(\eta + 2\sqrt{{\nu}})}{\nu} \leq \eta^2 + 2\eta = \delta,
\end{aligned}
\end{equation*}
where the last inequality uses $\nu\ge 1$. The bound \eqref{eq:scaling_matrix_bound} follows from the definition of the matrix norm induced by $H_x$.
\end{proof}

\subsection{Predictor--corrector algorithm}
We now present a predictor--corrector interior-point algorithm based on the conjugate-free scaling. Each iteration consists of a predictor step and a corrector step. The predictor step reduces the linear residual and decreases the homogeneous complementarity measure. However, it may move the iterate farther from the central path, and the corrector step is then used to restore centrality. 

\subsubsection{Predictor Step}
Given $z \in \mathcal{N}(\eta, \beta_l, \beta_u)$, define the predictor direction $\Delta z_p = (\Delta \bar{x}_p; \Delta y_p; \Delta \bar{s}_p)$ as the solution to the linear system
\begin{subequations}\label{eq:PredSystemW}
	\begin{align}
		G(\dy_p;\dx_p)-(0;\ds_p) &= -r(z), \label{eq:PredSystemW-a}\\
		\tau\Delta\kappa_p+	\kappa\Delta\tau_p&=-\kappa\tau,\label{eq:PredSystemW-c}\\
		\Delta s_p + \mu W\, \Delta x_p &= - s, \label{eq:PredSystemW-b}
	\end{align}
\end{subequations}
where $r(z) := G(y; \bar{x}) - (0; \bar{s})$ is the linear residual. The predictor step aims to reduce both the linear residual $r(z)$ and the
homogeneous complementarity measure $\bar\mu$. We set the step length $\alpha_p \in (0, 1]$ to ensure that $z^+ = z + \alpha_p \Delta z_p$ remains in $\mathcal{N}(\eta^+, \beta_l^+, \beta_u^+)$.

\subsubsection{Corrector Step}
Given $z^+ \in \mathcal{N}(\eta^+, \beta_l^+, \beta_u^+)$, define the corrector direction $\Delta z_c = (\Delta \bar{x}_c; \Delta y_c; \Delta \bar{s}_c)$ by solving
\begin{subequations}\label{eq:CorrSystemW}
	\begin{align}
		G(\dy_c;\dx_c)-(0;\ds_c) &= 0, \label{eq:CorrSystemW-a}\\
			{\tau^+\Delta\kappa_c+	\kappa^+\Delta\tau_c}&=\sigma\mu^+
			-\kappa^+\tau^+,\label{eq:CorrSystemW b}\\
		\Delta s_c +\mu^+ {W^+}\,\Delta x_c &= -\psi(x^+, s^+,\mu^+).
		\label{eq:CorrSystemW-b}
	\end{align}
\end{subequations}
Here, $\sigma \in [0, 1]$ is a centering parameter, $\mu^+=\frac{\langle x^+,s^+\rangle}{\nu}$ denotes the updated cone complementarity measure obtained after the predictor step, and the scaling matrix used in \eqref{eq:CorrSystemW} is recomputed as
\begin{equation*}\label{eq:update_w_pl}
W^+=H_{x^+}+\frac{s^+(s^+)^\top}{\nu(\mu^+)^2}-\frac{g_{x^+}g_{x^+}^\top}{\nu}.
\end{equation*}
The corrector step recenters the iterate toward the central path, reducing the complementarity residual $\psi(x^+, s^+, \mu^+)$ and ensuring that the next iterate $z^{++} = z^+ + \alpha_c \Delta z_c$ stays within $\mathcal N(\eta,\beta_l,\beta_u)$. We take one corrector step with a fixed step length $\alpha_c \in (0,1]$ to refine the iterate.

Algorithm~\ref{alg:pd_ipm} presents the resulting predictor--corrector interior-point method.
\begin{algorithm}[H]
\caption{The predictor--corrector interior-point method}
\label{alg:pd_ipm}
\begin{algorithmic}[1]
\STATE \textbf{Preprocessing:}
Express the problem in the standard conic form \eqref{eq:coneopt}.
Choose an initial point $z^{(0)}=(\bar x^{(0)};y^{(0)};\bar s^{(0)})\in
\mathcal N(\eta,\beta_l,\beta_u)$ as described in
Section~\ref{subsec:initialization}. Set $z=z^{(0)}$.
\FOR{$k = 0, 1, 2, \ldots $}
    \IF{a convergence criterion is satisfied}
        \RETURN $(\bar{x},y,\bar{s})$
    \ENDIF
    \STATE \textbf{Predictor step:}
    \STATE Solve the system \eqref{eq:PredSystemW} to obtain $\Delta z_p=\left(\dx_p;\dy_p;\ds_p\right)$.
    \STATE Set the step length $\alpha_{p}$ and update $z^+=z+\alpha_p\Delta z_p$.
    \STATE \textbf{Corrector step:}
    \STATE Solve the system \eqref{eq:CorrSystemW} to obtain $\Delta z_c=\left(\dx_c;\dy_c;\ds_c\right)$.
    \STATE Set the step length $\alpha_{c}$ and update $z^{++}=z^{+}+\alpha_c\Delta z_c$.
    \STATE Set $z\leftarrow z^{++}$.
\ENDFOR
\end{algorithmic}
\end{algorithm}

The convergence analysis of this algorithm, including the precise choice of parameters $(\eta,\beta_l,\beta_u,\alpha_p,\alpha_c,\sigma)$ and the resulting iteration complexity bound $\mathcal{O}\left(\sqrt{\nu}\log(1/ \varepsilon)\right)$, is provided in Section~\ref{sec4}.

\section{Complexity analysis}\label{sec4}
This section establishes the global iteration complexity of the proposed primal--dual interior-point method for nonsymmetric conic optimization.
We proceed in three stages: first, we analyze the properties of the predictor step; second, we examine the corrector step; and finally, we prove neighborhood preservation for one predictor--corrector iteration, from which the final complexity estimate follows.

\subsection{Properties of the predictor step}
We begin by establishing norm bounds for $s$ within the neighborhood $\mathcal N(\eta,\beta_l,\beta_u)$.
\begin{lemma}\label{lem:s-size}
	Let $z\in\mathcal N(\eta,\beta_l,\beta_u)$. Then
	\begin{align}\label{eq:s-size}
		\| s\|_{ x}^*
		\le (\sqrt{\nu}+\eta)\,\mu \quad\text{and}\quad
		\| s\|_{W}^*
		\le \frac{(\sqrt{\nu}+\eta)}{\sqrt{1-\delta}}\,\mu.
	\end{align}
\end{lemma}
\begin{proof}
	Write $ s=-\mu g_{ x}+\psi( x, s,\mu)$. Combining \eqref{eq:neighborhood} with Lemma~\ref{lem:loghom-id} gives
	\[
	\| s\|_{ x}^*
	\le \mu\|g_{ x}\|_{ x}^*+\|\psi\|_{ x}^*
	\le \mu\sqrt{\nu}+\eta\mu.
	\]
	The second inequality of \eqref{eq:s-size} follows from \Cref{Thm:Hessian bound}.
\end{proof}

The next result demonstrates that the predictor step contracts the linear residual by the factor $1-\alpha_p$, a property that will be used to control feasibility.
\begin{lemma}\label{lem:feas_reduction}
Let $z\in\mathcal N(\eta,\beta_l,\beta_u)$ and
$\Delta z_p=(\Delta\bar x_p;\Delta y_p;\Delta\bar s_p)$ be the solution of \eqref{eq:PredSystemW}. Set
$
   z^{+}:=z + \alpha_p\Delta z_p
$
for some $\alpha_p\in(0,1]$. Then
$r(z^+) = (1 - \alpha_p) r(z).$
\end{lemma}
\begin{proof}
The mapping $r(z)=G(y;\bar x)-(0;\bar s)$ is affine in $z$. Hence
\[
r(z+\alpha_p\dz_p)=r(z)+\alpha_p\bigl(G(\dy_p;\dx_p)-(0;\ds_p)\bigr).
\]
Using \eqref{eq:PredSystemW-a}, $G(\dy_p;\dx_p)-(0;\ds_p)=-r(z)$, which yields the claim.
\end{proof}

We now bound the predictor direction in various norms.
\begin{lemma}\label{eq:lemma bound of xs}
For $z \in \mathcal N(\eta,\beta_l,\beta_u)$, let $\dz_p$ solve \eqref{eq:PredSystemW}. Then:
\begin{enumerate}
    \item[(i)]$\langle\dx_p,\ds_p\rangle=0$;
    \item[(ii)] $
    \|\dxx_p\|_{{W}}\leq \gamma,\, \|\dss_p\|_{{W}}^*\leq \gamma \mu,
$
 where $\gamma:=
\sqrt{\frac{\left(\eta+\sqrt{\nu}\right)^2}{1-\delta}+\frac{\beta_u}{2}}$. Furthermore, $
    \|\dxx_p\|_{{x}}\leq \frac{\gamma}{\sqrt{1-\delta}},\, \|\dss_p\|_{{x}}^*\leq \sqrt{1+\delta}\,\gamma\mu.
$
\end{enumerate}
\end{lemma}
\begin{proof}
    (i) From \eqref{eq:PredSystemW-a} and the linearity of $G$, we have $G(y+\dy_p;\bar{x}+\dx_p)-(0;\bar{s}+\ds_p) = 0$. Taking the inner product of this identity with $(y+\dy_p;\bar{x}+\dx_p)$ and using the skew-symmetry of $G$ gives
    $
        \langle \bar{x}+\dx_p,\bar{s}+\ds_p\rangle = 0,
    $
    which rearranges to 
     \begin{equation}\label{eq:dxds}
        \langle\dx_p,\ds_p\rangle=-\left(\langle \bar{s},\dx_p\rangle+\langle \bar{x},\ds_p\rangle+\langle \bar{s},\bar{x}\rangle\right).
    \end{equation}
    Substitute $\dss_p=-{s}-\mu W \dxx_p$ 
    into \eqref{eq:dxds}:
    \begin{equation*}\label{eq:dxds2}
        \begin{aligned}
        \langle\dx_p,\ds_p\rangle&=-\left(\langle {s},\dxx_p\rangle+\kappa\Delta\tau_p-\langle {x},{s}\rangle-\mu\langle {x},W\dxx_p\rangle+\tau\Delta\kappa_p +\langle \bar{s},\bar{x}\rangle\right)\\
        &=-\left(\langle {s},\dxx_p\rangle-\mu\langle W {x},\dxx_p\rangle-\kappa\tau-\langle {x},{s}\rangle+\langle \bar{s},\bar{x}\rangle\right)=0.
        \end{aligned}
    \end{equation*}

         (ii) From \eqref{eq:PredSystemW-c}, it follows that $\Delta\tau_p=-\tau-\frac{\tau}{\kappa}	\Delta\kappa_p$, and thus \begin{equation}\label{eq:bound of tkp}\Delta\tau_p\Delta\kappa_p=-\tau \Delta\kappa_p-\frac{\tau}{\kappa}	(\Delta\kappa_p)^2\leq \frac{\tau\kappa}{4}.\end{equation} By part (i), we have $\langle\dx_p,\ds_p\rangle=\langle\Delta x_p,\Delta s_p\rangle+\Delta\tau_p\Delta\kappa_p=0$, which implies that $\Delta\tau_p\Delta\kappa_p=-\langle\Delta x_p,\Delta s_p\rangle$.
         Multiplying \eqref{eq:PredSystemW-b} by $W^{-\frac{1}{2}}$, taking norms, and then squaring both sides yields
        \begin{equation*}
        \begin{aligned}
            (\|\dss_p\|_{W}^{*})^2 + \mu^2\|\dxx_p\|_{W}^2 &= (\|{s}\|_{W}^{*})^2-2\mu \, \langle\dxx_p, \dss_p\rangle 
            = (\|{s}\|_{W}^{*})^2+2\mu\Delta\tau_p\Delta\kappa_p\\
            &\leq (\|{s}\|_{W}^{*})^2+\frac{\tau\kappa\mu}{2}\leq (\|{s}\|_{W}^{*})^2+\frac{\beta_u {\mu}^2}{2}.
            \end{aligned}
        \end{equation*}
        This result, together with \Cref{lem:s-size}, implies that
        \begin{equation*}
            \begin{aligned}
                (\|\dxx_p\|_{W})^2&\leq(\frac{\|{s}\|_{W}^{*}}{\mu})^2+\frac{\beta_u }{2}\leq \frac{\left(\eta+\sqrt{\nu}\right)^2}{1-\delta}+\frac{\beta_u}{2},\\
                (\|\dss_p\|_{W}^{*})^2&\leq (\|{s}\|_{W}^{*})^2+\frac{\beta_u {\mu}^2}{2}\leq \left(\frac{\left(\eta+\sqrt{\nu}\right)^2}{1-\delta}+\frac{\beta_u}{2}\right)\mu^2.
            \end{aligned}
        \end{equation*}
        Thus, $ \|\dxx_p\|_{W} \leq \gamma$ and $\|\dss_p\|_{W}^{*}\leq \gamma\mu$.
        Using Theorem~\ref{Thm:Hessian bound}, we obtain
        \[
        \begin{aligned}
         \|\dxx_p\|_{{x}}&\leq\frac{1}{\sqrt{1-\delta}}\|\dxx_p\|_{W}\leq \frac{\gamma}{\sqrt{1-\delta}},\\
         \|\dss_p\|_{{x}}^*&\leq\sqrt{1+\delta}\|\dss_p\|_{W}^{*}\leq \sqrt{1+\delta}\,\gamma\mu.
        \end{aligned}
        \]
         This proves the lemma.
\end{proof}

Having established bounds on the predictor direction, we next quantify how the  complementarity measures evolve under a predictor step. 
\begin{theorem}\label{thm:gap_evolution}
For $ z \in \mathcal N(\eta,\beta_l,\beta_u) $, let $ z^{+} = z + \alpha_{p} \dz_{p}$. Then
\begin{subequations}
\begin{align}
\left(1-\alpha_p-\frac{\gamma^2\alpha_p^2}{\nu}\right)\mu&\leq   \mu^+\leq \left(1-\alpha_p+\frac{\gamma^2\alpha_p^2}{\nu}\right)\mu,\label{eq:bound of mup}\\
|\mu^+-\mu|&\leq\alpha_p\mu\left(1+\frac{\gamma^2\alpha_p}{\nu}\right),\\
 \bar{\mu}^+&=(1-\alpha_p)\bar{\mu}.\label{ineq:barmu p}
 \end{align}
\end{subequations}
\end{theorem}
\begin{proof}
By definition, we have
$
\mu^{+}
=\frac1\nu\langle  x+\alpha_p\dxx_p,\  s+\alpha_p\dss_p\rangle.
$
Substituting $\dss_p=- s-\mu W\dxx_p$ gives 
\begin{equation}\label{eq:mu-expansion2}
\begin{aligned}
\mu^{+}
=&\,
(1-\alpha_p)\mu
+\frac{\alpha_p}{\nu}\Bigl(\langle \dxx_p, s\rangle-\langle  x,\mu W\dxx_p\rangle\Bigr)
+\frac{\alpha_p^{2}}{\nu}\langle\dxx_p,\dss_p\rangle \\
=&\,(1-\alpha_p)\mu
+\frac{\alpha_p}{\nu}\Bigl(\langle \dxx_p, s\rangle-\langle  \mu Wx,\dxx_p\rangle\Bigr)+\frac{\alpha_p^{2}}{\nu}\langle\dxx_p,\dss_p\rangle\\
 =&(1-\alpha_p)\mu+\frac{\alpha_p^{2}}{\nu}\langle\dxx_p,\dss_p\rangle.
\end{aligned}
\end{equation}
Here, the second equality uses the symmetry of $W$, while the last equality follows from the relation $s=\mu Wx$.
By the Cauchy--Schwarz inequality in the $W$-norm and Lemma~\ref{eq:lemma bound of xs}(ii), we have
\begin{equation*}
-\gamma^2\mu\leq-\|\dxx_p\|_W\|\dss_p\|_W^*\leq\langle\dxx_p,\dss_p\rangle\leq \|\dxx_p\|_W\|\dss_p\|_W^*\leq \gamma^2\mu.
\end{equation*}
Substituting this estimate into \eqref{eq:mu-expansion2} yields the two-sided bound
\begin{equation}\label{eq:aux equa 1}
  \left(1-\alpha_p-\frac{\gamma^2\alpha_p^2}{\nu}\right)\mu\leq   \mu^+\leq \left(1-\alpha_p+\frac{\gamma^2\alpha_p^2}{\nu}\right)\mu.
\end{equation}
In particular, it follows from \eqref{eq:aux equa 1} that
\begin{equation*}
    |\mu^+-\mu|\leq\alpha_p\mu\left(1+\frac{\gamma^2\alpha_p}{\nu}\right).
\end{equation*}
From the predictor equations \eqref{eq:PredSystemW-c}--\eqref{eq:PredSystemW-b} and $\mu W x=s$, we have
\begin{equation}\label{eq:first order sum}
\begin{aligned}
\,\langle \dx_p,\bar s\rangle+\langle \bar x,\ds_p \rangle
=&\,\langle \dxx_p, s\rangle+\langle  x,\dss_p\rangle + \kappa\Delta \tau+ \tau\Delta\kappa\\
=&\,\langle \dxx_p, s\rangle+\langle  x,-s-\mu W\dxx_p\rangle -\tau\kappa\\
=&\,-\langle x,s\rangle- \tau\kappa
=-\bar{\nu}\bar{\mu}.
\end{aligned}
\end{equation}
Therefore, using \eqref{eq:first order sum} and Lemma~\ref{eq:lemma bound of xs}(i), we derive
\begin{equation*}
\begin{aligned}
\bar{\mu}^+ &=\, \frac{\langle  \bar{x}+\alpha_p\dx_p,\  \bar{s}+\alpha_p\ds_p\rangle}{\bar{\nu}}\\
&=\,\frac{\langle\bar{x},\bar{s}\rangle+\alpha_p(\langle \dx_p,\bar s\rangle+\langle \bar x,\ds_p \rangle)+\alpha_p^2(\langle\dx_p,\ds_p\rangle)}{\bar{\nu}}\\
&=\,(1-\alpha_p)\bar{\mu}.
\end{aligned}
\end{equation*}
This proves the theorem.
\end{proof}

To ensure that the predictor point remains in the neighborhood $ \mathcal{N}(\eta^+,\beta_l^+,\beta_u^+)$, we must verify that the updated iterate satisfies the proximity condition with parameter $\eta^+$ as well as the scalar homogeneous complementarity bounds $\beta_l^+$ and $\beta_u^+$. The next two theorems provide these guarantees.
\begin{theorem}\label{thm:neighborhood of predictor}
For $z \in \mathcal{N}(\eta,\beta_l,\beta_u)$, let $\dz_p$ solve \eqref{eq:PredSystemW} and $z^+=z+\alpha_p \Delta z_p$ for some $\alpha_p>0$ such that $\omega:=\frac{\alpha_p\gamma}{\sqrt{1-\delta}}<1$ and $1-\omega-(1-\delta)\omega^2>0$. Then
\[
x^+\in \intK\quad\text{and} \quad
\|\psi( x^+, s^+,\mu^+)\|_{ x^+}^{*}\le \eta^+\mu^+,
\]
where $\eta^+:=\dfrac{\omega}{(1-\omega)^2}+ \frac{\eta+(1-\delta)(\omega+\omega^2)+\sqrt{1-\delta^2}\,\omega}{(1-\omega-(1-\delta)\omega^2)(1-\omega)}$.
Furthermore, if $\eta^+<1$, then $s^+\in \operatorname{int}(\K^*)$.
\end{theorem}
\begin{proof}
The bound on the predictor direction from Lemma~\ref{eq:lemma bound of xs} implies that
\[
        \|x^+-x\|_x
        =
        \alpha_p\|\Delta x_p\|_x
        \le
        \frac{\alpha_p\gamma}{\sqrt{1-\delta}}
        =
        \omega
        <1 .
\]
Since $x\in\operatorname{int}(\K)$, Lemma~\ref{lemma:perturb result} gives
$
x^+\in\operatorname{int}(\K).
$
We now estimate the residual $\psi^+:=\psi(x^+,s^+,\mu^+)$. Since $s^+=s+\alpha_p\Delta s_p$, the residual $\psi^+$ can be decomposed as
\[
\begin{aligned}
\psi^+
&=
s+\alpha_p\Delta s_p+\mu^+g_{x^+}        \\
&=
\psi(x,s,\mu)
+(\mu^+-\mu)g_x
+\mu^+(g_{x^+}-g_x)
+\alpha_p\Delta s_p .
\end{aligned}
\]
Taking the dual local norm at $x$ and applying the triangle inequality yield
\begin{equation}\label{ineq:auxi-ineq-2}
\begin{aligned}
\|\psi^+\|_x^*
\le
\|\psi(x,s,\mu)\|_x^*
+|\mu^+-\mu|\,\|g_x\|_x^*        
+\mu^+\|g_{x^+}-g_x\|_x^*
+\alpha_p\|\Delta s_p\|_x^* .
\end{aligned}
\end{equation}
The first term is controlled directly by 
$
\|\psi(x,s,\mu)\|_x^*\le\eta\mu .
$

We next bound the second term of \eqref{ineq:auxi-ineq-2}. By Theorem~\ref{thm:gap_evolution} and
Lemma~\ref{lem:loghom-id},
\begin{equation}\label{ineq:auxi-ineq-1}
\begin{aligned}
|\mu^+-\mu|\,\|g_x\|_x^*
\le
\alpha_p\mu
\left(1+\frac{\gamma^2\alpha_p}{\nu}\right)\sqrt{\nu}      
=
\alpha_p\sqrt{\nu}\,\mu
+
\frac{\gamma^2\alpha_p^2}{\sqrt{\nu}}\mu .
\end{aligned}
\end{equation}
The definition of $\omega$, together with
$\gamma\ge \sqrt{\nu/(1-\delta)}$, gives the two estimates
$
        \alpha_p\sqrt{\nu}
        \le
        (1-\delta)\omega,
         \,
        \frac{\gamma^2\alpha_p^2}{\sqrt{\nu}}
        \le
       (1-\delta)\omega^2 .
$
Substituting these two estimates into \eqref{ineq:auxi-ineq-1} gives
\[
        |\mu^+-\mu|\,\|g_x\|_x^*
        \le
        (1-\delta)(\omega+\omega^2)\mu .
\]
The third term of \eqref{ineq:auxi-ineq-2} is controlled by Lemma~\ref{lemma:perturb result}. Since
$\|x^+-x\|_x\le\omega$, we have
\[
        \|g_{x^+}-g_x\|_x^*
        \le
        \frac{\|x^+-x\|_x}{1-\|x^+-x\|_x}
        \le
        \frac{\omega}{1-\omega}.
\]
The last term of \eqref{ineq:auxi-ineq-2} is estimated by Lemma~\ref{eq:lemma bound of xs}:
\[
        \alpha_p\|\Delta s_p\|_x^*
        \le
        \alpha_p\sqrt{1+\delta}\,\gamma\mu
        =
        \sqrt{1-\delta^2}\,\omega\mu .
\]
Combining the preceding four bounds gives
\[
\begin{aligned}
\|\psi^+\|_x^*
&\le
\frac{\omega}{1-\omega}\mu^+
+
\Bigl[
\eta
+(1-\delta)(\omega+\omega^2)
+\sqrt{1-\delta^2}\,\omega
\Bigr]\mu .
\end{aligned}
\]
We now replace $\mu$ on the right-hand side by $\mu^+$. From Theorem~\ref{thm:gap_evolution} and the definitions of $\omega$ and $\gamma$, we have $\alpha_p\leq \frac{1-\delta}{\sqrt{\nu}}\omega\leq\omega$ and 
\[
\mu^+
\ge
\left(
1-\alpha_p-\frac{\gamma^2\alpha_p^2}{\nu}
\right)\mu                                      
\ge
\bigl(1-\omega-(1-\delta)\omega^2\bigr)\mu .
\]
Therefore,
\begin{equation}\label{ineq:auxi-ineq-4}
\|\psi^+\|_x^*
\le
\left[
\frac{\omega}{1-\omega}
+
\frac{
\eta
+(1-\delta)(\omega+\omega^2)
+\sqrt{1-\delta^2}\,\omega
}
{
1-\omega-(1-\delta)\omega^2
}
\right]\mu^+.
\end{equation}
The dual local norm at $x^+$ is now obtained from Lemma~\ref{lemma:perturb result}:
\begin{equation}\label{ineq:auxi-ineq-3}
        \|\psi^+\|_{x^+}^*
        \le
        \frac{1}{1-\|x^+-x\|_x}\|\psi^+\|_x^*
        \le
        \frac{1}{1-\omega}\|\psi^+\|_x^* .
\end{equation}
Combining \eqref{ineq:auxi-ineq-3} with \eqref{ineq:auxi-ineq-4} gives
$
        \|\psi^+\|_{x^+}^*
        \le
        \eta^+\mu^+ .
$

It remains to prove the strict dual feasibility of $s^+$ under the condition
$\eta^+<1$. Since $x^+\in\operatorname{int}(\K)$, Theorem~\ref{thm:dual interior} gives
$
        -g_{x^+}\in\operatorname{int}(\K^*)
$
and
\[
\left\|
        \frac{s^+}{\mu^+}+g_{x^+}
\right\|_{-g_{x^+}}
=
\frac{1}{\mu^+}\|\psi^+\|_{x^+}^*
\le
\eta^+
<1 .
\]
Applying Lemma~\ref{lemma:perturb result} to the conjugate barrier $F^*$ at the point
$-g_{x^+}$, we obtain
$
        \frac{s^+}{\mu^+}\in\operatorname{int}(\K^*) .
$
Since $\mu^+>0$, this implies
$
        s^+\in\operatorname{int}(\K^*) .
$
Therefore, the proof is complete.
\end{proof}
\begin{theorem}\label{thm:bound of etatau}
Let $z\in\mathcal N(\eta,\beta_l,\beta_u)$ and $0<\beta_l<1$. If $\alpha_p \in (0,1]$ satisfies $\alpha_p< \frac{\sqrt{\beta_l^2+4\beta_l\gamma^2}-\beta_l}{2\gamma^2}$, then
\[\beta_l^+\mu^+\leq\tau^+\kappa^+\leq \beta_u^+\mu^+,\]
where $\beta_l^+:=\frac{(1-\alpha_p)\beta_l-\alpha_p^2\gamma^2}{1-\alpha_p+{\gamma^2\alpha_p^2}}$ and $\beta_u^+:=\frac{(1-\alpha_p+\alpha_p^2/4)\beta_u}{1-\alpha_p-{\gamma^2\alpha_p^2}}$. Furthermore, $\tau^+>0$ and $\kappa^+>0$.
\end{theorem}
\begin{proof}
We first derive the lower bound for $\tau^+\kappa^+$. Expanding the product yields
\begin{equation}\label{eq:aux_in_thm4.6}
\begin{aligned}
\tau^+\kappa^+
&=
\tau\kappa
+\alpha_p(\kappa\Delta\tau_p+\tau\Delta\kappa_p)
+\alpha_p^2\Delta\kappa_p\Delta\tau_p  \\
&=
(1-\alpha_p)\tau\kappa
-\alpha_p^2\langle\dxx_p,\dss_p\rangle ,
\end{aligned}
\end{equation}
where the last equality uses Lemma~\ref{eq:lemma bound of xs}(i). Since
$z\in\mathcal N(\eta,\beta_l,\beta_u)$, we have
$\tau\kappa\geq\beta_l\mu$. Moreover, by the Cauchy--Schwarz inequality in the $W$-norm and Lemma~\ref{eq:lemma bound of xs}(ii), 
\[
\langle\dxx_p,\dss_p\rangle
\leq
\|\dxx_p\|_W\|\dss_p\|_W^*
\leq
\gamma^2\mu .
\]
Consequently,
\[
\begin{aligned}
\tau^+\kappa^+
\geq
(1-\alpha_p)\beta_l\mu
-\alpha_p^2\gamma^2\mu  
=
\left((1-\alpha_p)\beta_l-\alpha_p^2\gamma^2\right)\mu .
\end{aligned}
\]
The step length $\alpha_p\in(0, \frac{\sqrt{\beta_l^2+4\beta_l\gamma^2}-\beta_l}{2\gamma^2})$ satisfies
$\gamma^2\alpha_p^2+\beta_l\alpha_p-\beta_l<0$. 
Therefore, $(1-\alpha_p)\beta_l-\alpha_p^2\gamma^2>0$ and
\begin{equation}\label{eq:lower bound of taukappa}
\begin{aligned}
\tau^+\kappa^+
&\geq
\left((1-\alpha_p)\beta_l-\alpha_p^2\gamma^2\right)
\left(1-\alpha_p+\frac{\gamma^2\alpha_p^2}{\nu}\right)^{-1}
\mu^+  \\
&\geq
\frac{(1-\alpha_p)\beta_l-\alpha_p^2\gamma^2}
{1-\alpha_p+\gamma^2\alpha_p^2}
\mu^+ ,
\end{aligned}
\end{equation}
where the first inequality uses the upper estimate in \eqref{eq:bound of mup}, and the last inequality follows from $\nu\geq 1$.

We next establish the upper bound. Since $z\in\mathcal N(\eta,\beta_l,\beta_u)$, we have $\tau\kappa\leq\beta_u\mu$. Combining this inequality with \eqref{eq:aux_in_thm4.6} and \eqref{eq:bound of tkp} yields
\[
\begin{aligned}
\tau^+\kappa^+
\leq
(1-\alpha_p)\tau\kappa
+\frac{\alpha_p^2}{4}\tau\kappa  
=
\left(1-\alpha_p+\frac{\alpha_p^2}{4}\right)\tau\kappa\leq
\left(1-\alpha_p+\frac{\alpha_p^2}{4}\right)\beta_u\mu .
\end{aligned}
\]
Combining the preceding inequality with the lower bound in \eqref{eq:bound of mup} and using $\nu\geq 1$, we obtain
\[
\begin{aligned}
\tau^+\kappa^+
\leq
\left(1-\alpha_p+\frac{\alpha_p^2}{4}\right)
\frac{\beta_u\mu^+}
{1-\alpha_p-\frac{\gamma^2\alpha_p^2}{\nu}}  
\leq
\frac{(1-\alpha_p+\alpha_p^2/4)\beta_u}
{1-\alpha_p-\gamma^2\alpha_p^2}
\mu^+ ,
\end{aligned}
\]
where the last inequality holds because $1-\alpha_p-\gamma^2\alpha_p^2>(1-\alpha_p)\beta_l-\alpha_p^2\gamma^2>0$.

It remains to prove the positivity of $\tau^+$ and $\kappa^+$. If $0<\alpha_p< \frac{\sqrt{\beta_l^2+4\beta_l\gamma^2}-\beta_l}{{2} \gamma^2}$, then we have  $\tau^{+}\kappa^{+}>0$ by \eqref{eq:lower bound of taukappa}. Consequently, either $\tau^{+}>0$ and $\kappa^{+}>0$, or $\tau^{+}<0$ and $\kappa^{+}<0$. Suppose, for the sake of contradiction, that $\tau^{+}<0$ and $\kappa^{+}<0$. Since the mapping $\alpha \mapsto \tau+\alpha\,\Delta\tau_p$ is continuous and $\tau>0$, the intermediate value theorem guarantees the existence of some $\hat\alpha\in(0,\alpha_p)$ such that $\tau+\hat\alpha\Delta\tau_p=0$. For this $\hat\alpha$, we have
\[
(\tau+\hat\alpha\Delta\tau_p)\,(\kappa+\hat\alpha\Delta\kappa_p)=0,
\]
whereas the lower bound \eqref{eq:lower bound of taukappa} is strictly positive for all $\hat\alpha\in(0,\alpha_p)$. This contradiction implies that $\tau^{+}>0$ and $\kappa^{+}>0$.
\end{proof}

\subsection{Properties of the corrector step}
Having characterized the predictor step, we now turn to the corrector step. The corrector step aims to restore proximity to the central path without altering the linear residual, thereby maintaining the progress achieved by the predictor. We begin by updating the scaling matrix and establishing the associated bounds.
\begin{theorem}\label{Thm:Hessian bound corrector}
For $z^+\in \mathcal{N}(\eta^+,\beta_l^+,\beta_u^+)$, let $W^+ = H_{x^+} + \frac{s^+(s^+)^\top}{\nu(\mu^+)^2}-\frac{g_{x^+} g_{x^+}^\top}{\nu}$. If $0<\eta^+<\sqrt2-1$, then 
\begin{equation}
    (1-\delta^+)\,H_{{x^+}}\ \preceq\ W^+\ \preceq\ (1+\delta^+)\,H_{{x^+}},
\end{equation}
where $\delta^+:=(\eta^+)^2+2\eta^+<1$.
\end{theorem}
\begin{proof}
The proof is similar to that of Theorem \ref{Thm:Hessian bound}.
\end{proof}

A fundamental property of the corrector step is that it preserves the linear residual achieved by the predictor.
\begin{lemma}\label{lem:r-corr}
Let $z^+\in\mathcal N(\eta^+,\beta_l^+,\beta_u^+)$ and $\Delta z_c=(\Delta\bar x_c;\Delta y_c;\Delta\bar s_c)$ be the solution of \eqref{eq:CorrSystemW}. Set $z^{++}:=z^+ + \alpha_c\Delta z_c$ for some $\alpha_c\in(0,1]$. Then $r(z^{++})=r(z^{+}).$
\end{lemma}
\begin{proof}
	Since $r(\cdot)$ is affine,
	$r(z^{+}+\alpha_c\dz_c)=r(z^{+})+\alpha_c(G(\dy_c;\dx_c)-(0;\ds_c)).$
	By \eqref{eq:CorrSystemW-a}, the increment vanishes.
\end{proof}
\begin{lemma}\label{lemma:bound of dtdk}
Let $z^+\in\mathcal N(\eta^+,\beta_l^+,\beta_u^+)$ and $\Delta z_c$ be the solution of \eqref{eq:CorrSystemW}. Define $\theta:=\max\left\{\frac{(\beta_l^+ - \sigma)^2}{4\beta_l^+},\frac{(\beta_u^+ - \sigma)^2}{4\beta_u^+}\right\}.$ Then $\Delta\tau_c  \Delta\kappa_c\leq {\theta}\,\mu^+$.
\end{lemma}
\begin{proof}
Using \eqref{eq:CorrSystemW b} to eliminate $\Delta\kappa_c$, we obtain
\begin{align*}
    \Delta\tau_c  \Delta\kappa_c
    &= -\frac{\kappa^+}{\tau^+} (\Delta\tau_c)^2 + \frac{\sigma\mu^+ - \tau^+\kappa^+}{\tau^+} \Delta\tau_c
    \leq \frac{(\sigma\mu^+ - \tau^+\kappa^+)^2}{4\tau^+\kappa^+}.
    \label{eq:quadratic_function}
\end{align*}
Set $\zeta := \frac{\kappa^+\tau^+}{\mu^+} \in [\beta_l^+, \beta_u^+]$ and define the function $h(\zeta) = \frac{(\sigma - \zeta)^2}{4\zeta}$. Substituting $\zeta = \frac{\kappa^+\tau^+}{\mu^+}$ into $h(\zeta)$ yields the identity
\[
\mu^+ h(\zeta) = \frac{(\sigma\mu^+ - \tau^+\kappa^+)^2}{4\tau^+\kappa^+}.
\]
Since $\zeta\in[\beta_l^+,\beta_u^+]$ and
$
h'(\zeta)=\frac{(\zeta-\sigma)(\zeta+\sigma)}{4\zeta^2},
$
the maximum of $h$ over this interval is attained at one of the endpoints, i.e.,
\[
h(\zeta) \leq \max\left\{h(\beta_l^+),h(\beta_u^+)\right\}.
\]
Combining the above inequalities, we obtain
\[
\Delta\tau_c  \Delta\kappa_c \leq \frac{(\sigma\mu^+ - \tau^+\kappa^+)^2}{4\tau^+\kappa^+} = \mu^+ h(\zeta) \leq \mu^+ \max\left\{\frac{(\beta_l^+ - \sigma)^2}{4\beta_l^+},\frac{(\beta_u^+ - \sigma)^2}{4\beta_u^+}\right\}.
\]
The desired inequality follows.
\end{proof}

As in the predictor-step analysis, we now establish norm bounds for the corrector direction.
\begin{lemma}\label{eq:lemma bound of xs_c}
For $z^+ \in  \mathcal{N}(\eta^+,\beta_l^+,\beta_u^+)$, let $\dz_c$ solve \eqref{eq:CorrSystemW}. Then,
\begin{enumerate}
    \item[(i)]$\langle\dx_c,\ds_c\rangle=0$;
    \item[(ii)] $ \|\dxx_c\|_{W^+} \leq \gamma^+$ and $\|\dss_c\|_{W^+}^*\leq  \gamma^+\mu^+$, where $\gamma^+:=\sqrt{\frac{(\eta^+)^2}{1-\delta^+}+2\theta}$.
    Furthermore,
    $
    \|\dxx_c\|_{{x^+}}\leq \frac{\gamma^+}{\sqrt{1-\delta^+}},\, \|\dss_c\|_{{x^+}}^*\leq\sqrt{1+\delta^+}\,\gamma^+ \mu^+ .
$
\end{enumerate}
\end{lemma}
\begin{proof}
    (i) Taking the inner product of \eqref{eq:CorrSystemW-a} with $(\Delta y_c,\dx_c)$ gives $\langle\dx_c,\ds_c\rangle=0$ since $G$ is skew-symmetric.
    
     (ii) Define $\psi^+:=\psi( x^+, s^+,\mu^+)$. Multiplying \eqref{eq:CorrSystemW-b} by $(W^+)^{-\frac{1}{2}}$, taking norms, and then squaring both sides give
    \begin{equation*}
    \begin{aligned}
        (\|\dss_c\|_{W^+}^*)^2+(\mu^+)^2\|\dxx_c\|_{W^+}^2&=(\|\psi^+\|_{W^+}^*)^2-2\mu^+\langle\dxx_c,\dss_c\rangle.
        \end{aligned}
    \end{equation*}
    Using Theorem~\ref{Thm:Hessian bound corrector}, together with the relation $\langle\dxx_c,\dss_c\rangle=-\Delta\kappa_c\Delta\tau_c$, we obtain
\begin{equation*}
\begin{aligned}
(\|\dss_c\|_{W^+}^*)^2
+(\mu^+)^2\|\dxx_c\|_{W^+}^2
&\leq
\frac{(\|\psi^+\|_{x^+}^*)^2}{1-\delta^+}
+2\mu^+\Delta\kappa_c\Delta\tau_c  \\
&\leq
\frac{(\eta^+\mu^+)^2}{1-\delta^+}
+2\theta(\mu^+)^2  
=
(\gamma^+\mu^+)^2 .
\end{aligned}
\end{equation*}
    Consequently, we have \begin{equation}\label{eq:correct bound}
    \|\dxx_c\|_{W^+} \leq \gamma^+ \quad\text{and}\quad \|\dss_c\|_{W^+}^*\leq \gamma^+ \mu^+.\end{equation}
    It follows from Theorem~\ref{Thm:Hessian bound corrector} and \eqref{eq:correct bound} that
     \begin{equation*}
     \begin{aligned}
         \|\dxx_c\|_{x^+}&\leq\frac{1}{\sqrt{1-\delta^+}} \|\dxx_c\|_{W^+}\leq \frac{\gamma^+}{\sqrt{1-\delta^+}},\\
          \|\dss_c\|_{x^+}^*&\leq\sqrt{1+\delta^+} \|\dss_c\|_{W^+}^*\leq   \sqrt{1+\delta^+}\gamma^+ \mu^+.
         \end{aligned}
     \end{equation*}
     This completes the proof.
\end{proof}

We now quantify the evolution of the complementarity measures under the corrector step.
\begin{theorem}\label{thm:gap_evolution_corrector}
For $ z^+ \in \mathcal{N}(\eta^+,\beta_l^+,\beta_u^+) $ and $\sigma\geq\beta_l^+$, let $ z^{++} = z^+ + \alpha_{c} \dz_{c}$, $\mu^{++}:=\frac{\langle{x}^{++}, {s}^{++}\rangle}{{\nu}}$ and $\bar{\mu}^{++}:=\frac{\langle\bar{x}^{++}, \bar{s}^{++}\rangle}{\bar{\nu}}$. Then
\begin{subequations}
\begin{align}
\left(1-\frac{(\alpha_c\gamma^+)^2}{\nu}\right)\mu^+\leq\mu^{++}&\leq \left(1+\frac{(\alpha_c\gamma^+)^2}{\nu}\right)\mu^+,\label{ineq:mu pp}\\
 |\mu^{++}-\mu^{+}|&\leq\frac{(\alpha_c\gamma^+)^2}{\nu}\mu^+,\\
 \bar{\mu}^{++}&\leq\left(1+\frac{\alpha_c\sigma-\alpha_c\beta_l^+}{\nu}\right)\bar{\mu}^+.\label{ineq:barmu pp}
 \end{align}
\end{subequations}
\end{theorem}
\begin{proof}
 Let
$
        \psi^+ := \psi(x^+,s^+,\mu^+).
$
Since $\langle x^+,g_{x^+}\rangle=-\nu$, we have
\[
        \langle x^+,\psi^+\rangle
        =
        \langle x^+,s^+\rangle
        +\mu^+\langle x^+,g_{x^+}\rangle
        =
        \nu\mu^+-\nu\mu^+
        =0 .
\]
The corrector equation \eqref{eq:CorrSystemW-b}, together with the scaling relation $\mu^+W^+x^+=s^+$, gives
\begin{equation}\label{eq:relation_corr_xs}
        \langle x^+,\Delta s_c\rangle
        =
        -\mu^+\langle x^+,W^+\Delta x_c\rangle
        -\langle x^+,\psi^+\rangle
        =
        -\langle \mu^+W^+x^+,\Delta x_c\rangle= -\langle s^+,\Delta x_c\rangle.
\end{equation}
Consequently,
\[
\begin{aligned}
\mu^{++}
=
\frac{1}{\nu}
\langle x^+ +\alpha_c\Delta x_c,
        s^+ +\alpha_c\Delta s_c\rangle                        
=
\mu^+
+
\frac{\alpha_c^2}{\nu}\langle \Delta x_c,\Delta s_c\rangle .
\end{aligned}
\]
By the Cauchy--Schwarz inequality in the $W^+$-norm and Lemma~\ref{eq:lemma bound of xs_c}(ii),
\[
        |\langle \Delta x_c,\Delta s_c\rangle|
        \le
        \|\Delta x_c\|_{W^+}\|\Delta s_c\|_{W^+}^*
        \le
        (\gamma^+)^2\mu^+ .
\]
This proves the two-sided bound for $\mu^{++}$, and the estimate for $|\mu^{++}-\mu^+|$ follows immediately.

Expanding the definition of $\bar\mu^{++}$ gives
\[
\begin{aligned}
\bar\nu\bar\mu^{++}
=
\langle x^++\alpha_c\Delta x_c,
        s^++\alpha_c\Delta s_c\rangle                         
+
(\tau^+ +\alpha_c\Delta\tau_c)
(\kappa^+ +\alpha_c\Delta\kappa_c) .
\end{aligned}
\]
Using
$\langle \Delta\bar x_c,\Delta\bar s_c\rangle=0$, \eqref{eq:relation_corr_xs}
and \eqref{eq:CorrSystemW b},
we obtain
\[
\begin{aligned}
\bar\nu\bar\mu^{++}
=
\nu\mu^+ +\tau^+\kappa^+
+
\alpha_c(\sigma\mu^+-\tau^+\kappa^+)        
=
\bar\nu\bar\mu^+
+
\alpha_c(\sigma\mu^+-\tau^+\kappa^+) .
\end{aligned}
\]
Since $z^+\in\mathcal N(\eta^+,\beta_l^+,\beta_u^+)$, we have $\tau^+\kappa^+\ge \beta_l^+\mu^+$. Hence
\[
        \bar\nu\bar\mu^{++}
        \le
        \bar\nu\bar\mu^+
        +\alpha_c(\sigma-\beta_l^+)\mu^+ .
\]
Finally, because
$
        \bar\nu\bar\mu^+
        =
        \nu\mu^+ +\tau^+\kappa^+
        \ge \nu\mu^+ ,
$
we have $\mu^+\le \bar\nu\bar\mu^+/\nu$. Dividing by $\bar\nu$ then gives
\[
        \bar\mu^{++}
        \le
        \left(
        1+\frac{\alpha_c(\sigma-\beta_l^+)}{\nu}
        \right)\bar\mu^+ .
\]
This yields the desired conclusion.
\end{proof}
 
The proximity condition for the corrector step is established in the following theorem.
\begin{theorem}\label{Them:corrector-prox}
For $z^+ \in \mathcal{N}(\eta^+,\beta_l^+,\beta_u^+)$, let $\dz_c$ solve \eqref{eq:CorrSystemW} and $z^{++}=z^++\alpha_c \Delta z_c$ for some $0<\alpha_c\leq 1$ such that $\omega^+:=\frac{\alpha_c\gamma^+}{\sqrt{1-\delta^+}}{ < 1}$. Then $x^{++}\in \intK$ and 
\begin{equation}\label{eq:psi neighborhood}
\begin{aligned}
&\|\psi( x^{++}, s^{++},\mu^{++})\|_{ x^{++}}^{*}\leq \eta^{++} \mu^{++},
\end{aligned}
\end{equation}
where $\eta^{++}:=\left(1-{(\alpha_c\gamma^+)^2}\right)^{-1}\bigg[\Bigl(\frac{\te}{1-\te}\Bigr)^2
+\frac{(1-\alpha_c+\delta^+)\te}{\alpha_c(1-\te)}
+(1-\alpha_c){\tilde{\eta}}+(\alpha_c\gamma^+)^2\bigg]$ \\and $\tilde{\eta}:=\frac{\sqrt{1+\delta^+}\,\gamma^+}{1-\omega^+}$. Furthermore, if $\eta^{++}<1$, then $s^{++}\in \operatorname{int}(\K^*)$.
\end{theorem}
\begin{proof}
 Note that $\|x^{++}-x^+\|_{x^+}=\alpha_c\|\dxx_c\|_{x^+}\leq\frac{\alpha_c\gamma^+}{\sqrt{1-\delta^+}}{ < 1}$, and hence $x^{++}\in \intK$ by Lemma~\ref{lemma:perturb result}. 
Next, we define the function $f:\operatorname{int}(\K)\to\mathbb R$  by
\[
f(v):=\frac{1}{\mu^{+}}( s^+ +\dss_c)^\top v+F(v).
\]
Since $F$ is a self-concordant barrier on $\operatorname{int}(\K)$, the function $f$ is also a self-concordant barrier on $\operatorname{int}(\K)$ and its Hessian equals $H_v$.
Hence the exact Newton direction of $f$ at $v\in\operatorname{int}(\K)$ is
\[\label{eq:newton direction}
n_f(v):=-H_v^{-1}\nabla f(v)=-\frac{1}{\mu^{+}} H_v^{-1}\bigl( s^+ +\dss_c+\mu^{+} g_v\bigr).
\]
Now define the inexact Newton direction for $f$ at $x^+$ by
\[
\tilde{n}_f(x^+):=-(W^+)^{-1}\nabla f(x^+)=-\frac{1}{\mu^{+}} (W^+)^{-1}\bigl( s^+ +\dss_c+\mu^{+} g_{x^+}\bigr).
\]
Using $\psi^+ = s^+ + \mu^+ g_{x^+}$ and \eqref{eq:CorrSystemW-b},
we obtain
\[
\tilde{n}_f(x^+) = -\frac{1}{\mu^+}(W^+)^{-1}(\psi^+ + \Delta s_c) = \Delta x_c,
\]
and therefore
\[
x^{++} = x^+ + \alpha_c \Delta x_c = x^+ + \alpha_c \tilde{n}_f(x^+).
\]
Moreover, by Theorem~\ref{Thm:Hessian bound corrector}, we have $
(1 - \delta^+) H_{x^+} \preceq W^+ \preceq (1 + \delta^+) H_{x^+}.$ Thus Theorem~\ref{thm:inexact_newton} is applicable to the self-concordant barrier $f$ at the point $x^+$, with scaling matrix $W^+$, step length $ \alpha_c$, and
\[
\rho^+ := \alpha_c \|\tilde{n}_f(x^+)\|_{x^+} = \alpha_c \|\Delta x_c\|_{x^+} \leq \omega^+ < 1.
\]
The inexact Newton estimate then yields
\[
\|n_f(x^{++})\|_{x^{++}} \leq \left( \frac{\rho^+}{1 - \rho^+} \right)^2 + \frac{1 - \alpha_c + \delta^+}{1 - \rho^+} \|\tilde{n}_f(x^+)\|_{x^+}.
\]
Using the bounds $\rho^+ \leq \omega^+$ and
$
\|\tilde{n}_f(x^+)\|_{x^+} = \|\Delta x_c\|_{x^+} \leq \frac{\omega^+}{\alpha_c},
$
we further obtain
\[
\|n_f(x^{++})\|_{x^{++}} \leq \left( \frac{\omega^+}{1 - \omega^+} \right)^2 + \frac{(1 - \alpha_c + \delta^+)\omega^+}{\alpha_c(1 - \omega^+)}.
\]
By the definition of $n_f(x^{++})$, this estimate is equivalent to
\begin{equation}\label{eq:newton bound}
\frac{1}{\mu^+} \|s^+ + \Delta s_c + \mu^+ g_{x^{++}}\|_{x^{++}}^* \leq \left( \frac{\omega^+}{1 - \omega^+} \right)^2 + \frac{(1 - \alpha_c + \delta^+)\omega^+}{\alpha_c(1 - \omega^+)}. 
\end{equation}
Lemma~\ref{lemma:perturb result}, together with Lemma~\ref{eq:lemma bound of xs_c}(ii), gives
\begin{equation}\label{eq:bound of dss_c}
\|\dss_c\|_{ x^{++}}^*\leq \frac{\|\dss_c\|_{ x^{+}}^*}{1-\alpha_c\|\dxx_c\|_{ x^{+}}}\leq \frac{\sqrt{1+\delta^+}\,\gamma^+\mu^+}{1-\omega^+}={\tilde{\eta}\mu^+}.
\end{equation}

We are now ready to estimate \eqref{eq:psi neighborhood}. Let $\psi^{++}:=\psi(x^{++},s^{++},\mu^{++}) $. By adding and subtracting $s^+ +\Delta s_c+\mu^+g_{x^{++}}$, we have
\[
\begin{aligned}
\psi^{++}
&=
s^+ +\alpha_c\Delta s_c+\mu^{++}g_{x^{++}}        \\
&=
\bigl(s^+ +\Delta s_c+\mu^+g_{x^{++}}\bigr)
-(1-\alpha_c)\Delta s_c
+(\mu^{++}-\mu^+)g_{x^{++}} .
\end{aligned}
\]
Taking the dual norm at $x^{++}$ and dividing by $\mu^{++}$, we estimate the three terms on the right-hand side separately. The identity $\|g_{x^{++}}\|^*_{x^{++}}=\sqrt{\nu}$, together with \eqref{eq:newton bound} and \eqref{eq:bound of dss_c}, leads to
\[
\begin{aligned}
\frac{\|\psi^{++}\|_{x^{++}}^*}{\mu^{++}}
\le
\frac{\mu^+}{\mu^{++}}
\bigg[
\left(\frac{\omega^+}{1-\omega^+}\right)^2
+
\frac{(1-\alpha_c+\delta^+)\omega^+}
     {\alpha_c(1-\omega^+)}
+
(1-\alpha_c)\widetilde\eta
\bigg]                                      
+
\frac{|\mu^{++}-\mu^+|}{\mu^{++}}\sqrt{\nu}.
\end{aligned}
\]
Using Theorem~\ref{thm:gap_evolution_corrector} and $\nu\ge1$, we further obtain
\begin{small}
\[
\begin{aligned}
\frac{\|\psi^{++}\|_{x^{++}}^*}{\mu^{++}}
&\le
\left(
1-\frac{(\alpha_c\gamma^+)^2}{\nu}
\right)^{-1}
\bigg[
\left(\frac{\omega^+}{1-\omega^+}\right)^2
+
\frac{(1-\alpha_c+\delta^+)\omega^+}
     {\alpha_c(1-\omega^+)}
+
(1-\alpha_c)\widetilde\eta
+
\frac{(\alpha_c\gamma^+)^2}{\sqrt\nu}
\bigg]                                      \\
&\le
\left(1-(\alpha_c\gamma^+)^2\right)^{-1}
\bigg[
\left(\frac{\omega^+}{1-\omega^+}\right)^2
+
\frac{(1-\alpha_c+\delta^+)\omega^+}
     {\alpha_c(1-\omega^+)}
+
(1-\alpha_c)\widetilde\eta
+
(\alpha_c\gamma^+)^2
\bigg].
\end{aligned}
\]
\end{small}
This proves
$
        \|\psi^{++}\|_{x^{++}}^*
        \le
        \eta^{++}\mu^{++}.
$

Since $\eta^{++}<1$, the same argument as in the proof of \Cref{thm:neighborhood of predictor} implies that $s^{++}\in \operatorname{int}(\K^*)$.
\end{proof}

\begin{theorem}\label{Thm:corrector-scale-bound}
Let $z^+\in\mathcal N(\eta^+,\beta_l^+,\beta_u^+)$. Suppose that $\omega^+<1$ and that $\alpha_c\in(0,1]$ satisfies $\alpha_c<\frac{\sqrt{(\sigma-\beta_l^+)^2+4\beta_l^+(\gamma^+)^2}+\sigma-\beta_l^+}{2(\gamma^+)^2}$. Then
\[
\begin{aligned}
\beta_l^{++}\mu^{++}
\leq\tau^{++}\kappa^{++}\leq\beta_u^{++}\mu^{++},
\end{aligned}
\]
where $$
\begin{aligned}
\beta_l^{++}&:=\left(1+\frac{(\alpha_c\gamma^+)^2}{\nu}\right)^{-1}\left[(1-\alpha_c)\beta_l^+-\alpha_c^2(\gamma^+)^2{ +\sigma\alpha_c}\right],\\ \beta_u^{++}&:=\left(1-\frac{(\alpha_c\gamma^+)^2}{\nu}\right)^{-1}\left[(1-\alpha_c)\beta_u^++\alpha_c^2\theta+\sigma\alpha_c\right].
\end{aligned}
$$
Furthermore, $\tau^{++}>0$ and $\kappa^{++}>0$.
\end{theorem}
\begin{proof}
Expanding $\tau^{++}\kappa^{++}$ and using \eqref{eq:CorrSystemW b} gives 
\begin{equation}\label{eq:aux-eq-1}
\begin{aligned}
\tau^{++}\kappa^{++}
&=
\tau^+\kappa^+
+\alpha_c(\kappa^+\Delta\tau_c+\tau^+\Delta\kappa_c)
+\alpha_c^2\Delta\tau_c\Delta\kappa_c        \\
&=
(1-\alpha_c)\tau^+\kappa^+
+\sigma\alpha_c\mu^+
+\alpha_c^2\Delta\tau_c\Delta\kappa_c .
\end{aligned}
\end{equation}
Applying Lemma~\ref{lemma:bound of dtdk} and the upper bound
$\tau^+\kappa^+\le \beta_u^+\mu^+$ to \eqref{eq:aux-eq-1} yields
\[
\begin{aligned}
\tau^{++}\kappa^{++}
&\le
\left[
(1-\alpha_c)\beta_u^+
+\sigma\alpha_c
+\alpha_c^2\theta
\right]\mu^+        \\
&\le
\left(
1-\frac{(\alpha_c\gamma^+)^2}{\nu}
\right)^{-1}
\left[
(1-\alpha_c)\beta_u^+
+\sigma\alpha_c
+\alpha_c^2\theta
\right]\mu^{++}        \\
&=
\beta_u^{++}\mu^{++},
\end{aligned}
\]
where the second inequality uses the lower bound for $\mu^{++}$ in
Theorem~\ref{thm:gap_evolution_corrector}. This upper bound is valid since $1-(\alpha_c\gamma^+)^2/\nu>0$, which follows from $\omega^+<1$.

For the lower bound, Lemma~\ref{eq:lemma bound of xs_c} gives
$
        \Delta\tau_c\Delta\kappa_c
        =
        -\langle \Delta x_c,\Delta s_c\rangle .
$
Hence, by the Cauchy--Schwarz inequality in the $W^+$-norm,
\begin{equation}\label{eq:4.13aux-ineq-2}
        \Delta\tau_c\Delta\kappa_c
        \ge
        -\|\Delta x_c\|_{W^+}\|\Delta s_c\|_{W^+}^*
        \ge
        -(\gamma^+)^2\mu^+ .
\end{equation}
Combining \eqref{eq:4.13aux-ineq-2} with
$\tau^+\kappa^+\ge \beta_l^+\mu^+$ and using \eqref{eq:aux-eq-1}, we obtain
\[
\begin{aligned}
\tau^{++}\kappa^{++}
&\ge
\left[
(1-\alpha_c)\beta_l^+
+\sigma\alpha_c
-(\alpha_c\gamma^+)^2
\right]\mu^+        \\
&\ge
\left[
(1-\alpha_c)\beta_l^+
+\sigma\alpha_c
-(\alpha_c\gamma^+)^2
\right]
\left(
1+\frac{(\alpha_c\gamma^+)^2}{\nu}
\right)^{-1}
\mu^{++}        \\
&=
\beta_l^{++}\mu^{++},
\end{aligned}
\]
where the second inequality uses the upper bound for $\mu^{++}$ in
Theorem~\ref{thm:gap_evolution_corrector}. The coefficient $(1-\alpha_c)\beta_l^+
+\sigma\alpha_c
-(\alpha_c\gamma^+)^2$ is positive under the step-length condition $\alpha_c\in(0,\frac{\sqrt{(\sigma-\beta_l^+)^2+4\beta_l^+(\gamma^+)^2}+\sigma-\beta_l^+}{2(\gamma^+)^2})$.

Finally, the step-length condition ensures that $\tau^{++}\kappa^{++}>0$. The same argument as in the proof of \Cref{thm:bound of etatau} shows that $\tau^{++}>0$ and $\kappa^{++}>0$.
\end{proof}

\subsection{Complexity results}
We now combine the predictor and corrector analyses to establish the main complexity results. The following theorem demonstrates that a fixed choice of algorithmic parameters ensures that the iterate returns to the initial neighborhood after each iteration.

\begin{theorem}\label{thm:guarantee neighborhood}
Fix the constants
\[
\eta=0.02,\quad \beta_l=0.9,\quad \beta_u=0.905,\quad \omega=0.005,\quad \alpha_c=0.85,\quad \sigma=0.9025.
\]
Let $\delta=\eta^2+2\eta$ and define
$\gamma=
\sqrt{\frac{(\eta+\sqrt{\nu})^2}{1-\delta}
+\frac{\beta_u}{2}}$ as in \Cref{eq:lemma bound of xs}.
Choose
$\alpha_p=\frac{\omega\sqrt{1-\delta}}{\gamma}$
so that the condition in \Cref{thm:neighborhood of predictor} is satisfied.
Let $z\in \mathcal{N}(\eta,\beta_l,\beta_u)$. Then:
\begin{enumerate}
    \item[(i)] The predictor point satisfies
    \[
    z^{+}\in 
    \mathcal{N}(0.03518,\,0.89995,\,0.90503).
    \]
    \item[(ii)] The corrected point satisfies
    \[
    z^{++}\in 
    \mathcal{N}(0.01665,\,0.90028,\,0.90376).
    \]
    Consequently, $z^{++}\in\mathcal{N}(\eta,\beta_l,\beta_u).$
\end{enumerate}
\end{theorem}
\begin{proof}
The proof is given in Appendix~\ref{secA2}. The code for reproducing the numerical bounds is available at \url{https://github.com/wenhfu/IPM-CF-check}.
\end{proof}

With neighborhood preservation established, we can now deduce the main complexity result. The analysis implies that our method achieves an $\mathcal{O}\left(\sqrt{\nu}\log(1/\varepsilon)\right)$ iteration bound.
\begin{theorem}
Let $z^{(0)}\in\mathcal N(\eta,\beta_l,\beta_u)$ be the initial point constructed in Section~\ref{subsec:initialization} and fix the parameter choice in Theorem~\ref{thm:guarantee neighborhood}. For any $\varepsilon \in (0,1)$, Algorithm~\ref{alg:pd_ipm} terminates in $\mathcal{O}\left(\sqrt{\nu}\log(1/\varepsilon)\right)$ iterations, yielding an iterate $z^{(N)}$ such that
$$\bar{\mu}^{(N)} \leq \varepsilon\,\bar\mu^{(0)} \quad \text{and} \quad \|r(z^{(N)})\| \leq \varepsilon\,\|r(z^{(0)})\|.$$
\end{theorem}

\begin{proof}
We first introduce several constants used in the following analysis.  Let $\widehat\beta_l^+:=0.89995 $ denote the uniform lower bound for $\beta_l^+$ over all $\nu\ge 1$. Define
\[
        c_p
        :=
        \frac{\omega(1-\delta)}
        {\sqrt{(1+\eta)^2+(1-\delta)\beta_u/2}}>0.0039511,
        \quad
        c_\mu
        :=
        \alpha_c(\sigma-\widehat\beta_l^+)=0.0021675,
\]
and set $\xi:=c_p-c_\mu>0.0017836.$ Thus $\xi$ is an absolute constant independent of $\nu$.

We next derive a lower bound on the predictor step length. Recall that $\alpha_p=\frac{\omega\sqrt{1-\delta}}{\gamma} .$ Since $\nu\ge1$, we have $(\eta+\sqrt{\nu})^2\le(1+\eta)^2\nu.$ Together with the definition of $\gamma$, this yields
\begin{equation}\label{eq:thm14-aux-1}
\begin{aligned}
\gamma^2
=
\frac{(\eta+\sqrt{\nu})^2}{1-\delta}
+
\frac{\beta_u}{2}                                      
\le
\frac{(1+\eta)^2+(1-\delta)\beta_u/2}{1-\delta}\,\nu .
\end{aligned}
\end{equation}
Substituting \eqref{eq:thm14-aux-1} into the formula for $\alpha_p$ gives
$\alpha_p\ge\frac{c_p}{\sqrt{\nu}}.$

We then estimate the decrease of $\bar\mu^{(k+1)}$ at the $k$-th iteration ($k\ge0$).
Theorems~\ref{thm:gap_evolution} and~\ref{thm:gap_evolution_corrector} imply that one predictor--corrector iteration satisfies
\[
\begin{aligned}
\bar\mu^{(k+1)}
&\le
\left(
1+\frac{\alpha_c(\sigma-\beta_l^+)}{\nu}
\right)
(1-\alpha_p)\bar\mu^{(k)} .
\end{aligned}
\]
Since $\beta_l^+\ge \widehat\beta_l^+$, the last
inequality implies
\begin{equation}\label{eq:thm15-aux-2}
\begin{aligned}
\bar\mu^{(k+1)}
\le
\left(
1+\frac{c_\mu}{\nu}
\right)
(1-\alpha_p)\bar\mu^{(k)}
\le
\left(
1-\alpha_p+\frac{c_\mu}{\nu}
\right)\bar\mu^{(k)}.
\end{aligned}
\end{equation}

Using the lower bound $\alpha_p\ge c_p/\sqrt{\nu}$ and $1/\nu\le 1/\sqrt{\nu}$, \eqref{eq:thm15-aux-2} becomes
\begin{equation}\label{eq:thm15-aux-3}
\begin{aligned}
\bar\mu^{(k+1)}
\le
\left(
1-\alpha_p+\frac{c_\mu}{\nu}
\right)\bar\mu^{(k)}                                      
\le
\left(
1-\frac{c_p-c_\mu}{\sqrt{\nu}}
\right)\bar\mu^{(k)}                                      
=
\left(
1-\frac{\xi}{\sqrt{\nu}}
\right)\bar\mu^{(k)} .
\end{aligned}
\end{equation}

The linear residual is contracted by the predictor step and preserved by the
corrector step. By Lemmas~\ref{lem:feas_reduction} and~\ref{lem:r-corr},
\begin{equation*}\label{eq:mu-decay}
        \|r(z^{(k+1)})\|
        \le
        (1-\alpha_p)\|r(z^{(k)})\| .
\end{equation*}
Since $\xi\le c_p$, the bound $\alpha_p\geq c_p/\sqrt{\nu}$ also gives
\begin{equation}\label{eq:r-decay}
        \|r(z^{(k+1)})\|
        \le
        \left(
        1-\frac{\xi}{\sqrt{\nu}}
        \right)
        \|r(z^{(k)})\| .
\end{equation}
Repeated application of \eqref{eq:thm15-aux-3} and \eqref{eq:r-decay} over $N$ predictor--corrector iterations yields
\[
        \bar\mu^{(N)}
        \le
        \left(
        1-\frac{\xi}{\sqrt{\nu}}
        \right)^N
        \bar\mu^{(0)},
        \quad
        \|r(z^{(N)})\|
        \le
        \left(
        1-\frac{\xi}{\sqrt{\nu}}
        \right)^N
        \|r(z^{(0)})\| .
\]
Using the standard estimate $(1-t)^N\le e^{-tN}$ for $t\in(0,1)$, it is
sufficient to choose
$
N\ge\frac{\sqrt{\nu}}{\xi}\log(1/\varepsilon).
$
This choice guarantees
\[
\bar\mu^{(N)}\le
\varepsilon\bar\mu^{(0)},
\quad
\|r(z^{(N)})\|\le\varepsilon\|r(z^{(0)})\| .
\]
The stated iteration bound follows.
\end{proof}

This result improves upon the $\mathcal{O}\left({\nu}\log(1/\varepsilon)\right)$ bound established by Badenbroek and Dahl \cite{badenbroek2022algorithm} for multi-secant BFGS scaling, and achieves the best-known complexity order for interior-point methods applied to nonsymmetric conic optimization. 

\section{Numerical experiments}\label{sec5}

In this section, we evaluate the proposed method, denoted by IPM-CF, on two families of nonsymmetric conic optimization problems. The first family is based on the operator perspective epigraph (OPE) cone
\[
\K_{\rm OPE}
=
\operatorname{cl}
\left\{
(T,X,Y)\in{\mathbb S}^n\times{\mathbb S}^n_{++}\times{\mathbb S}^n_{++}
:
T\succeq X^{1/2}[-\log(X^{-1/2}YX^{-1/2})]X^{1/2}
\right\},
\]
and the second family is based on the quantum relative entropy (QRE) cone
\[
\K_{\rm QRE}
=
\operatorname{cl}
\left\{
(t,X,Y)\in{\mathbb R}\times{\mathbb S}^n_{++}\times{\mathbb S}^n_{++}
:
t\ge \operatorname{tr}\bigl(X\log X-X\log Y\bigr)
\right\}.
\]
These cones are suitable for evaluating the proposed method because their barriers involve matrix logarithms, whose derivative evaluations can become expensive as the matrix dimension increases. We compare IPM-CF with QICS, which is designed for conic models arising in quantum information and provides direct support for both cone families. Several general-purpose nonsymmetric conic solvers, including Alfonso and Clarabel, do not directly support these formulations. We therefore use QICS as the main competing solver.

All experiments were performed on a desktop computer with an Intel Core i7-12700K CPU at 3.60 GHz and 16 GB of RAM under Windows 11. Both solvers were run with the same target accuracy. A run was declared successful when the relative duality gap, the relative primal infeasibility, and the relative dual infeasibility were all below $10^{-8}$. For each cone family and each matrix-dimension group, we generated a set of random feasible instances. The data were generated so that strictly feasible primal and dual points were available, and the right-hand sides and objective vectors were then constructed from these points. In total, the test set contains 80 OPE instances and 110 QRE instances.

\begin{figure}[h]
    \centering
    \includegraphics[width=0.98\textwidth]{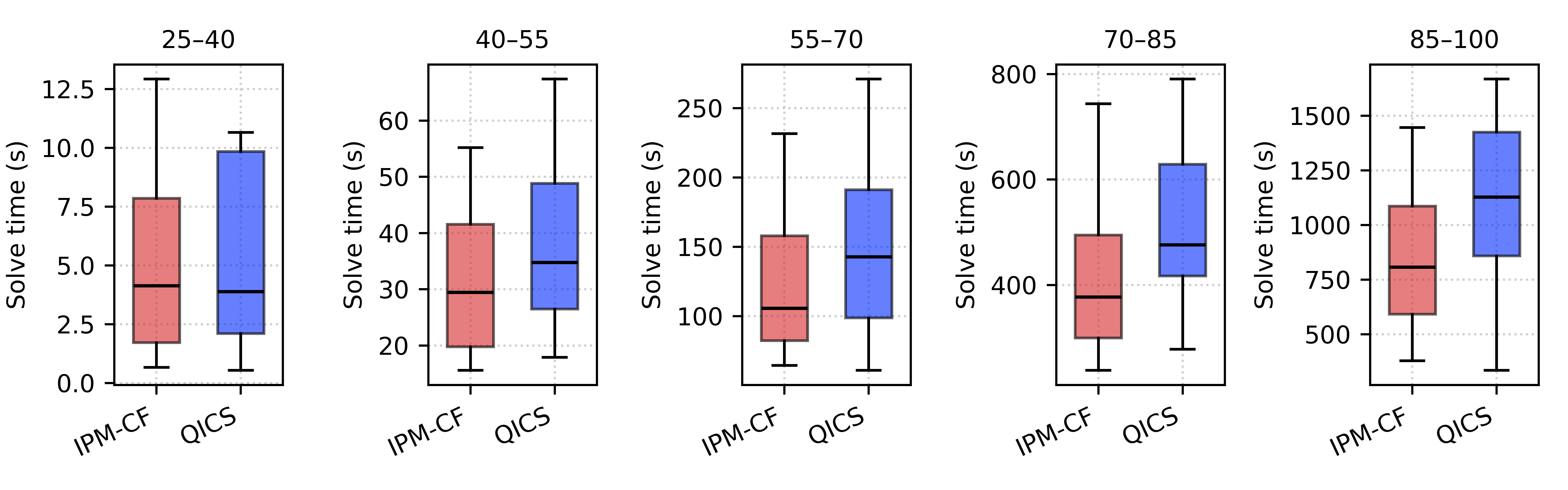}
    \caption{Running-time boxplots for OPE instances grouped by the matrix dimension.}
    \label{fig:ope_box_size}
\end{figure}
Figure~\ref{fig:ope_box_size} reports the running-time boxplots for the OPE instances. The results are grouped by the matrix dimension and show that IPM-CF generally has a smaller median running time than QICS. This advantage becomes more visible in the groups with larger matrix dimensions, where matrix logarithms and their derivatives are more expensive to evaluate. These results indicate that IPM-CF is typically faster on the tested instances.

\begin{figure}[h]
    \centering
    \includegraphics[width=0.98\textwidth]{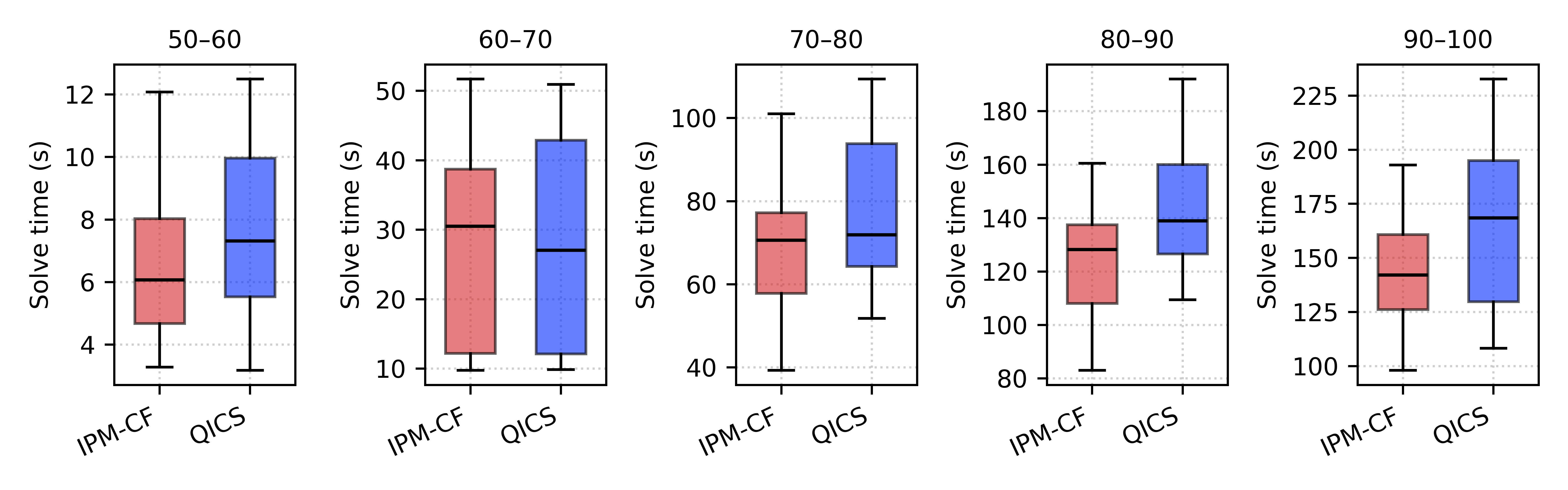}
    \caption{Running-time boxplots for QRE instances grouped by the matrix dimension.}
    \label{fig:qre_large_size}
\end{figure}
Figure~\ref{fig:qre_large_size} shows the running times for QRE instances. The comparison indicates that IPM-CF remains consistently competitive with QICS on this test set. In most dimension groups, the median running time of IPM-CF is lower, and the difference becomes more apparent for the higher-dimensional groups. This suggests that IPM-CF is also effective on the QRE instances.

\begin{figure}[!h]
\centering
\includegraphics[width=0.48\textwidth]{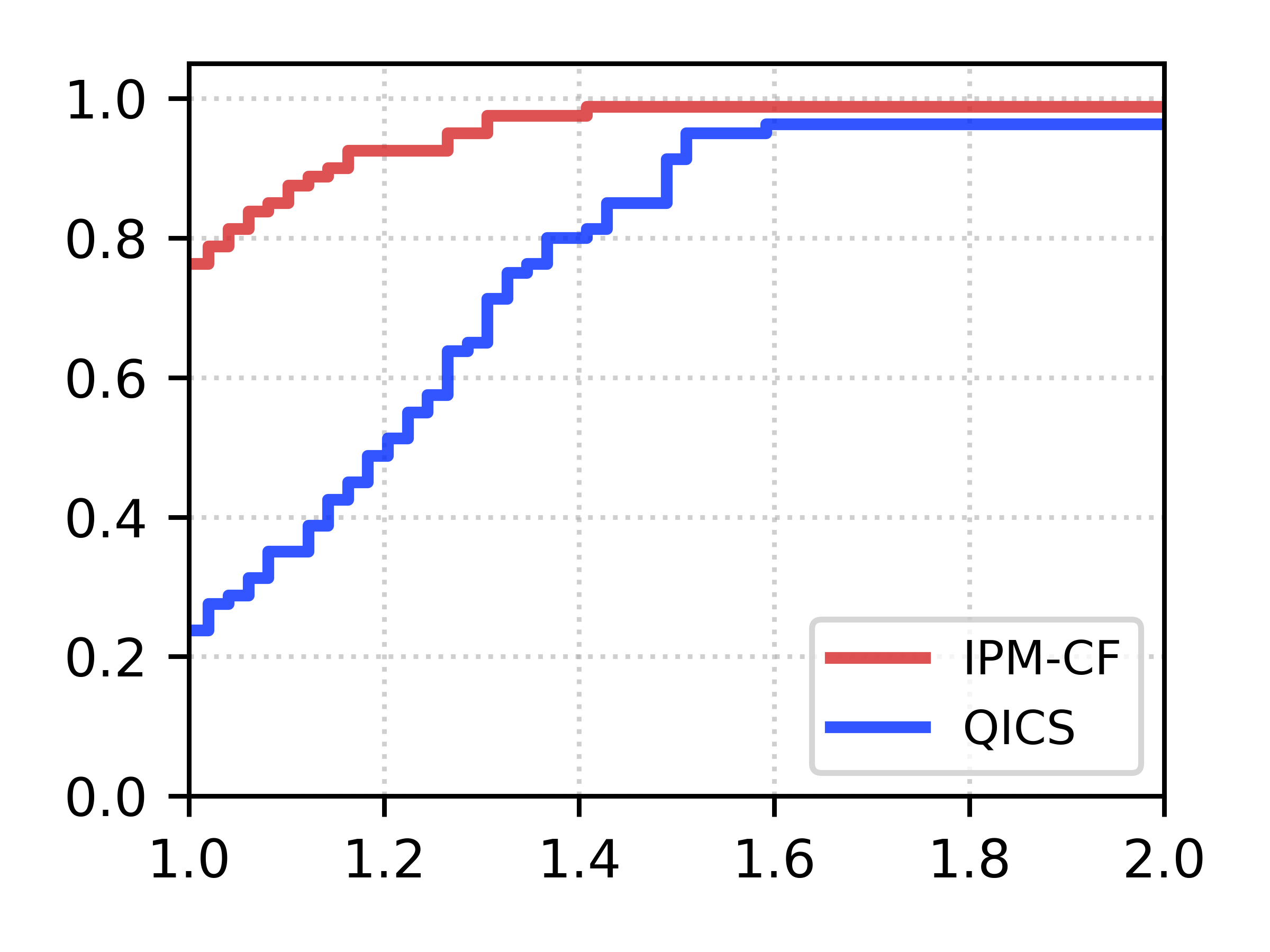}
\hfill 
\includegraphics[width=0.48\textwidth]{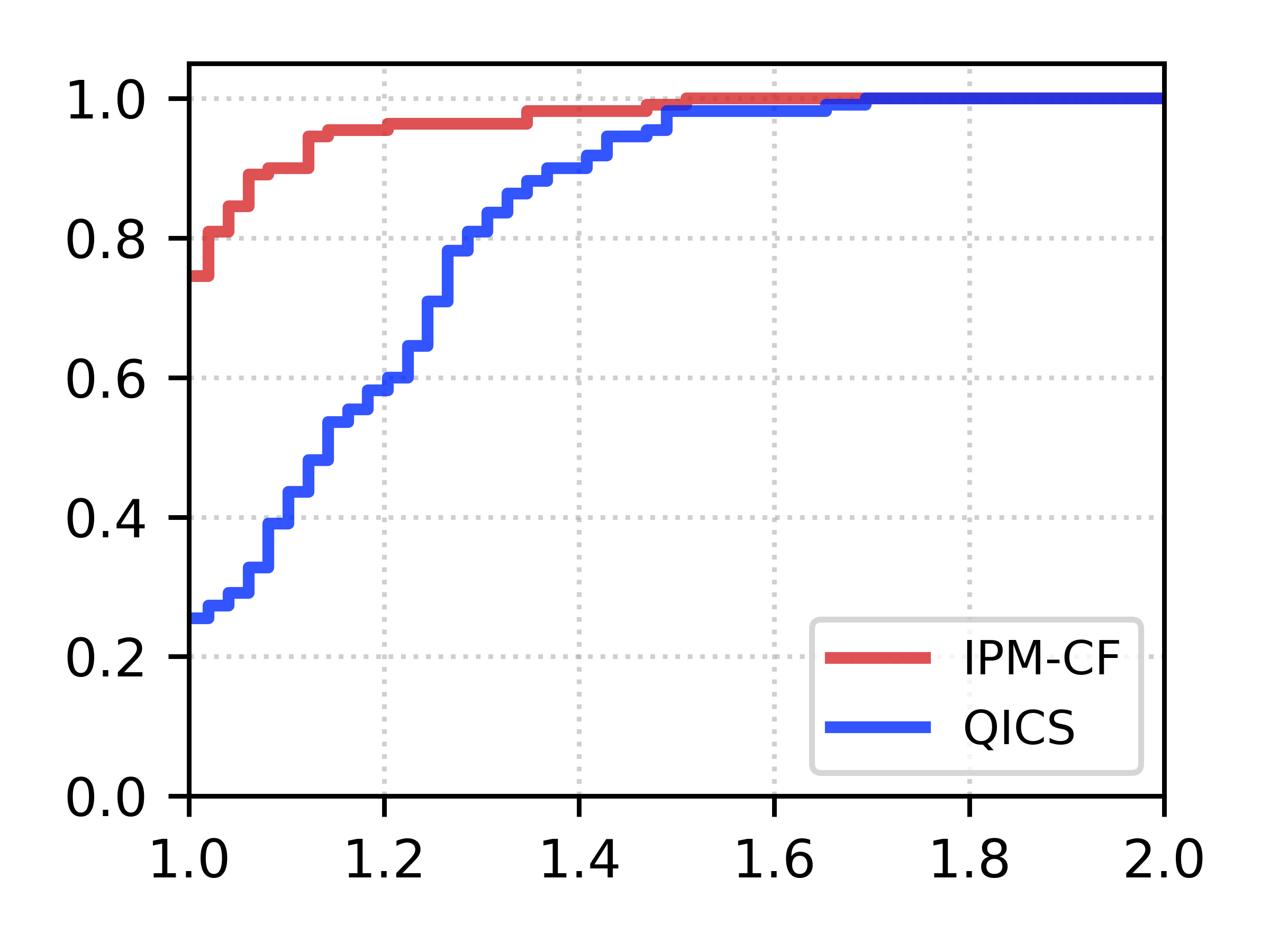}
\caption{Performance profiles based on total running time. Left: OPE instances. Right: QRE instances.}
\label{fig:performance profile}
\end{figure}
Figure~\ref{fig:performance profile} presents the performance profiles \cite{dolan2002benchmarking} based on total running time for the OPE and QRE instances. On both test sets, IPM-CF achieves a higher proportion of fastest solves, meaning it outperforms QICS on a larger fraction of instances. Its curves also remain above those of QICS over most of the displayed range, indicating that IPM-CF is not only more frequently the fastest solver but also competitive on the remaining instances.

\section{Conclusion}\label{sec6}
In this paper, we have proposed and analyzed a primal--dual interior-point method for nonsymmetric conic optimization. By avoiding conjugate-barrier derivatives, the proposed method removes a major computational bottleneck in high-dimensional nonsymmetric cones, where such derivatives are often unavailable in closed form or too costly to evaluate. The resulting scaling matrix is incorporated into a homogeneous self-dual predictor--corrector framework. We have also introduced a split central-path neighborhood that controls the conic variables and the scalar homogeneous variables separately. Within this neighborhood, the scaling matrix is uniformly comparable to the primal barrier Hessian. This comparison bound allows us to control the predictor and corrector directions and establish neighborhood preservation. As a result, we have obtained a worst-case iteration bound of $\mathcal{O}(\sqrt{\nu}\log(1/\varepsilon))$. This improves the previous $\mathcal{O}(\nu\log(1/\varepsilon))$ bound for multi-secant BFGS scalings and matches the best-known complexity order for interior-point methods applied to nonsymmetric conic optimization. The numerical results further show that the proposed method is competitive with QICS on instances involving the operator perspective epigraph cone and the quantum relative entropy cone.

%\iffalse
\appendix
\section{Proof of Theorem~\ref{thm:inexact_newton}}\label{secA1}
\begin{proof}
Since $F$ is a self-concordant barrier and $\rho=\|x^+-x\|_x<1$,  Lemma~\ref{lemma:perturb result} implies that $x^+\in\mathbb{X}$.
Applying \eqref{eq:sc_hessian_comp} with $u=x^+$ gives
\begin{equation}\label{eq:h_inv_h_bound}
\|H_{x^+}^{-1}H_x\|_x^{\rm op} \le \frac{1}{(1-\rho)^2}.
\end{equation}
Let $\tilde u:=H_x^{-1}g_{x^+}$. Using $n(x^+)=-H_{x^+}^{-1}g_{x^+}$, we have
\[
\|n(x^+)\|_{x^+}^2=
\langle g_{x^+},H_{x^+}^{-1}g_{x^+}\rangle =
\langle H_x \tilde u,H_{x^+}^{-1}H_x \tilde u\rangle. 
\]
By the property of the induced operator norm \cite[Section~2.3.1]{golub2013matrix}, we have 
\begin{equation}\label{eq:matrix analysis} 
\|Tv\|_x \le \|T\|_x^{\rm op}\|v\|_x . 
\end{equation}
Applying \eqref{eq:matrix analysis} with $T=H_{x^+}^{-1}H_x$ and $v=\tilde u$, we obtain 
\[ 
\|(H_{x^+}^{-1}H_x)\tilde u\|_x \le \|H_{x^+}^{-1}H_x\|_x^{\rm op}\|\tilde u\|_x . 
\]
Therefore, by the Cauchy--Schwarz inequality in the $H_x$-inner product,
\[
\|n(x^+)\|_{x^+}^2
\le
\|\tilde u\|_x\,
\|(H_{x^+}^{-1}H_x) \tilde u\|_x  
\le
\|H_{x^+}^{-1}H_x\|_x^{\rm op}\,
\|\tilde u\|_x^2  
=
\|H_{x^+}^{-1}H_x\|_x^{\rm op}\,
\bigl(\|g_{x^+}\|_x^*\bigr)^2 .
\]
Taking square roots and using \eqref{eq:h_inv_h_bound} yields
\begin{equation}\label{eq:nplus_reduction}
\|n(x^+)\|_{x^+} \le \frac{\|g_{x^+}\|_x^*}{1-\rho}.
\end{equation}

It remains to bound $\|g_{x^+}\|_x^*$. By the integral form of Taylor's formula for the gradient, we have
\[
g_{x^+}=\Big(g_x+\alpha H_x\tilde n(x)\Big)
+\int_0^1\big(\nabla^2 F(x+t\alpha \tilde n(x))-H_x\big)\,\alpha \tilde n(x)\,dt.
\]
Taking the dual norm $\|\cdot\|_x^*$ and using the triangle inequality gives
\begin{equation}\label{eq:gplus_split}
\|g_{x^+}\|_x^*
\le
\|g_x+\alpha H_x\tilde n(x)\|_x^*
+\int_0^1\|(\nabla^2 F(x+t\alpha \tilde n(x))-H_x)\,\alpha \tilde n(x)\|_x^*\,dt.
\end{equation}

We bound the first term in \eqref{eq:gplus_split} using the relative approximation bound \eqref{eq:relative_H_approx}.
Since $W\tilde n(x)=-g_x$, we have $g_x+H_x\tilde n(x)=(H_x-W)\tilde n(x)$, and therefore
\[
g_x+\alpha H_x\tilde n(x)=(H_x-W)\tilde n(x)-(1-\alpha)H_x\tilde n(x).
\]
Consequently,
\begin{equation}\label{eq:first_term_triangle}
\begin{aligned}
\|g_x+\alpha H_x\tilde n(x)\|_x^*
&\le
\|(H_x-W)\tilde n(x)\|_x^*+(1-\alpha)\|H_x\tilde n(x)\|_x^*\\
&=
\|(H_x-W)\tilde n(x)\|_x^*+(1-\alpha)\|\tilde n(x)\|_x.
\end{aligned}
\end{equation}
From \eqref{eq:relative_H_approx} we obtain $-\delta H_x\preceq W-H_x\preceq \delta H_x$, equivalently
\[
\|H_x^{-1/2}(H_x-W)H_x^{-1/2}\|_2\le \delta.
\]
Thus, for any $u$,
\[
\|(H_x-W)u\|_x^*=\|H_x^{-1/2}(H_x-W)u\|_2
\le \delta\|H_x^{1/2}u\|_2=\delta\|u\|_x,
\]
and in particular $\|(H_x-W)\tilde n(x)\|_x^*\le \delta\|\tilde n(x)\|_x$.
Substituting this into \eqref{eq:first_term_triangle} yields
\begin{equation}\label{eq:first_term_bound}
\|g_x+\alpha H_x\tilde n(x)\|_x^*\le (1-\alpha+\delta)\,\|\tilde n(x)\|_x.
\end{equation}

We now bound the integral term in \eqref{eq:gplus_split}.
Using $\|u\|_x^*=\|H_x^{-1}u\|_x$ and \eqref{eq:matrix analysis},
\[
\begin{aligned}
\|(\nabla^2 F(x+t\alpha \tilde n(x))-H_x)\,\alpha \tilde n(x)\|_x^*
&=
\|(H_x^{-1}\nabla^2 F(x+t\alpha \tilde n(x))-I)\,\alpha \tilde n(x)\|_x\\
&\le
\alpha\|\tilde n(x)\|_x\cdot \|H_x^{-1}\nabla^2 F(x+t\alpha \tilde n(x))-I\|_x^{\rm op}.
\end{aligned}
\]
For each $t\in[0,1]$, we have $\|t\alpha \tilde n(x)\|_x=t\rho<1$, so the second inequality in \eqref{eq:sc_hessian_comp} gives
\[
\|H_x^{-1}\nabla^2 F(x+t\alpha \tilde n(x))-I\|_x^{\rm op}
\le
\frac{1}{(1-t\rho)^2}-1.
\]
Therefore,
\begin{small}
\begin{align}\label{eq:integral_bound}
\int_0^1\|(\nabla^2 F(x+t\alpha \tilde n(x))-H_x)\,\alpha \tilde n(x)\|_x^*\,dt
&\le
\alpha\|\tilde n(x)\|_x\int_0^1\left(\frac{1}{(1-t\rho)^2}-1\right)dt \nonumber\\
&=
\rho\left(\frac{1}{1-\rho}-1\right)
=
\frac{\rho^2}{1-\rho}.
\end{align}
\end{small}
Combining \eqref{eq:gplus_split}, \eqref{eq:first_term_bound}, and \eqref{eq:integral_bound} yields
\begin{small}
\begin{equation}\label{eq:gplus_bound}
\|g_{x^+}\|_x^*
\le
\frac{\rho^2}{1-\rho}
+(1-\alpha+\delta)\,\|\tilde n(x)\|_x.
\end{equation}
\end{small}
Finally, substituting \eqref{eq:gplus_bound} into \eqref{eq:nplus_reduction} gives
\begin{small}
\[
\|n(x^+)\|_{x^+}
\le
\frac{1}{1-\rho}\left(\frac{\rho^2}{1-\rho}+(1-\alpha+\delta)\|\tilde n(x)\|_x\right)
=
\left(\frac{\rho}{1-\rho}\right)^2+\frac{(1-\alpha+\delta)\|\tilde n(x)\|_x}{1-\rho},
\]
\end{small}
which is \eqref{eq:main_bound}.
\end{proof}

\section{Proof of Theorem~\ref{thm:guarantee neighborhood}}\label{secA2}
\begin{proof}
Recall that
\[
\eta=0.02,\quad
\beta_l=0.9,\quad
\beta_u=0.905,\quad
\omega=0.005,\quad
\alpha_c=0.85,\quad
\sigma=0.9025.
\]
Then $\delta=\eta^2+2\eta=0.0404.$ Throughout the proof, quantities that depend on $\nu$, such as $\gamma(\nu)$ and $\alpha_p(\nu)$, are viewed as functions of $\nu$, whereas $\eta^+$ and $\delta^+$ are independent of $\nu$.

We first verify the predictor estimates. By definition,
\[
\gamma(\nu)^2=\frac{(\eta+\sqrt{\nu})^2}{1-\delta}+\frac{\beta_u}{2},\qquad \alpha_p(\nu)=\frac{\omega\sqrt{1-\delta}}{\gamma(\nu)}.
\]
The function $\gamma(\nu)$ is increasing and $\alpha_p(\nu)$ is decreasing on $[1,\infty)$. Moreover,
\[
\alpha_p(\nu)^2\gamma(\nu)^2=\omega^2(1-\delta)=:q=2.399\times10^{-5},
\]
so $q$ is independent of $\nu$. In particular,
\[
0<\alpha_p(\nu)\leq\alpha_p(1)<0.0039511211.
\]
The conditions $\omega<1$ and $1-\omega-(1-\delta)\omega^2>0$ in \Cref{thm:neighborhood of predictor} follow from
\[
\omega=0.005,\qquad 1-\omega-(1-\delta)\omega^2=0.99497601.
\]
The predictor step-length condition in \Cref{thm:bound of etatau} is equivalent to
\[
(1-\alpha_p(\nu))\beta_l-\alpha_p(\nu)^2\gamma(\nu)^2>0.
\]
Using the monotonicity of $\alpha_p(\nu)$ and the identity $\alpha_p(\nu)^2\gamma(\nu)^2=q$, we obtain
\[
(1-\alpha_p(\nu))\beta_l-q\geq(1-\alpha_p(1))\beta_l-q>0.8964.
\]
Hence all predictor step-length conditions hold for every $\nu\geq1$.

The scalar homogeneous complementarity bounds satisfy
\[
\beta_l^+(\nu)=\frac{(1-\alpha_p(\nu))\beta_l-q}{1-\alpha_p(\nu)+q},\qquad
\beta_u^+(\nu)=\frac{(1-\alpha_p(\nu)+\alpha_p(\nu)^2/4)\beta_u}{1-\alpha_p(\nu)-q}.
\]
For fixed $q$, differentiating the right-hand sides with respect to $\alpha_p$ gives
\[
\frac{\text{d}\beta_l^+}{\text{d}\alpha_p}=-\frac{(1+\beta_l)q}{(1-\alpha_p+q)^2}<0,\qquad
\frac{\text{d}\beta_u^+}{\text{d}\alpha_p}=\frac{\beta_u(2-\alpha_p)(\alpha_p+2q)}{4(1-\alpha_p-q)^2}>0.
\]
Because $\alpha_p(\nu)$ is decreasing, the chain rule shows that $\beta_l^+(\nu)$ is increasing and $\beta_u^+(\nu)$ is decreasing. Therefore,
\[
\begin{aligned}
\beta_l^+(\nu)\geq\beta_l^+(1)>0.8999542392>0.89995,\\
\beta_u^+(\nu)\leq\beta_u^+(1)<0.9050253438<0.90503.
\end{aligned}
\]
Since $\eta^+$ is independent of $\nu$, direct substitution of the fixed parameters gives $\eta^+<0.03518.$ It follows that
\[
z^+\in\mathcal N(0.03518,0.89995,0.90503).
\]

We next verify the corrector estimates. Since $\alpha_p(\nu)\to0$ as $\nu\to\infty$, the preceding formulas give
\[
\begin{aligned}
\lim_{\nu\to\infty}\beta_l^+(\nu)=\frac{\beta_l-q}{1+q}<0.8999544201<\sigma,\\
\lim_{\nu\to\infty}\beta_u^+(\nu)=\frac{\beta_u}{1-q}>0.9050217114>\sigma.
\end{aligned}
\]
Together with the monotonicity of $\beta_l^+(\nu)$ and $\beta_u^+(\nu)$, this yields
\[
\beta_l^+(\nu)<\sigma<\beta_u^+(\nu)\qquad\text{for all }\nu\geq1.
\]
Thus, the condition $\sigma\geq\beta_l^+(\nu)$ in \Cref{thm:gap_evolution_corrector} is satisfied uniformly.

The value of $\delta^+$ is independent of $\nu$ and satisfies
\[
\delta^+=(\eta^+)^2+2\eta^+<0.0715757738<1.
\]
For fixed $\sigma$, we have
\[
\frac{\text{d}}{\text{d}v}\left(\frac{(v-\sigma)^2}{4v}\right)=\frac{(v-\sigma)(v+\sigma)}{4v^2}.
\]
Hence $(v-\sigma)^2/(4v)$ is decreasing on $(0,\sigma)$ and increasing on $(\sigma,\infty)$. Since $\beta_l^+(\nu)$ increases while remaining below $\sigma$, the quantity
$
\frac{(\beta_l^+(\nu)-\sigma)^2}{4\beta_l^+(\nu)}
$
is decreasing in $\nu$. Similarly, since $\beta_u^+(\nu)$ decreases while remaining above $\sigma$,
$
\frac{(\beta_u^+(\nu)-\sigma)^2}{4\beta_u^+(\nu)}
$
is also decreasing. Therefore,
\[
\theta(\nu)=\max\left\{\frac{(\beta_l^+(\nu)-\sigma)^2}{4\beta_l^+(\nu)},\frac{(\beta_u^+(\nu)-\sigma)^2}{4\beta_u^+(\nu)}\right\}
\]
is decreasing on $[1,\infty)$.

Since $\gamma^+(\nu)^2=\frac{(\eta^+)^2}{1-\delta^+}+2\theta(\nu),$ the function $\gamma^+(\nu)$ is decreasing. Consequently, $\omega^+(\nu)=\frac{\alpha_c\gamma^+(\nu)}{\sqrt{1-\delta^+}}$ is also decreasing. Hence their largest values are attained at $\nu=1$, where
\[
\theta(1)<1.801\times10^{-6},\quad \gamma^+(1)<0.0365492393,\quad \omega^+(1)<0.0322421531<1.
\]

After substituting $\omega^+(\nu)=\alpha_c\gamma^+(\nu)/\sqrt{1-\delta^+}$ into the definition of $\eta^{++}(\nu)$, every term on the right-hand side is nondecreasing with respect to $\gamma^+(\nu)$ whenever $\omega^+(\nu)<1$. Since $\gamma^+(\nu)$ is decreasing, $\eta^{++}(\nu)$ is decreasing. Therefore,
\[
\eta^{++}(\nu)\leq\eta^{++}(1)<0.016640281<0.01665.
\]

Since $\alpha_c=0.85\in(0,1]$, the condition
$
\alpha_c
<
\frac{
\sqrt{(\sigma-\beta_l^+)^2
+4\beta_l^+(\gamma^+)^2}
+\sigma-\beta_l^+
}{
2(\gamma^+)^2
}
$
is equivalent to
\[
(1-\alpha_c)\beta_l^+(\nu)
+\sigma\alpha_c
-\bigl(\alpha_c\gamma^+(\nu)\bigr)^2
>0.
\]
The left-hand side is increasing in $\nu$ because
$\beta_l^+(\nu)$ is increasing and $\gamma^+(\nu)$ is decreasing.
Its value at $\nu=1$ is
\[
(1-\alpha_c)\beta_l^+(1)
+\sigma\alpha_c
-\bigl(\alpha_c\gamma^+(1)\bigr)^2
>
0.9011529865>0.
\]
Hence the corrector step-length condition in \Cref{Thm:corrector-scale-bound} holds for every
$\nu\ge1$.

For the lower scalar homogeneous complementarity, write
\[
\beta_l^{++}(\nu)=\frac{(1-\alpha_c)\beta_l^+(\nu)+\sigma\alpha_c-(\alpha_c\gamma^+(\nu))^2}{1+(\alpha_c\gamma^+(\nu))^2/\nu}.
\]
The numerator is increasing in $\nu$ because $\beta_l^+(\nu)$ is increasing and $\gamma^+(\nu)$ is decreasing. The denominator is decreasing because both $\gamma^+(\nu)$ and $1/\nu$ are decreasing. Thus, $\beta_l^{++}(\nu)$ is increasing, and
\[
\beta_l^{++}(\nu)\geq\beta_l^{++}(1)>0.9002840778>0.90028.
\]
Similarly, for the upper scalar homogeneous complementarity, write
\[
\beta_u^{++}(\nu)=\frac{(1-\alpha_c)\beta_u^+(\nu)+\sigma\alpha_c+\alpha_c^2\theta(\nu)}{1-(\alpha_c\gamma^+(\nu))^2/\nu}.
\]
Its numerator is decreasing because both $\beta_u^+(\nu)$ and $\theta(\nu)$ are decreasing. Its denominator is increasing and positive because $(\alpha_c\gamma^+(\nu))^2/\nu$ is decreasing and $\omega^+(\nu)<1$. Hence 
$\beta_u^{++}(\nu)$ is decreasing, and
\[
\beta_u^{++}(\nu)\leq\beta_u^{++}(1)<0.9037523584<0.90376.
\]
We conclude that
\[
z^{++}\in\mathcal N(0.01665,0.90028,0.90376)\subseteq\mathcal N(0.02,0.9,0.905).
\]
This proves the theorem.
\end{proof}
%\fi

\bibliographystyle{siamplain}
\bibliography{references}

@article{badenbroek2022algorithm,
  title={An algorithm for nonsymmetric conic optimization inspired by \text{MOSEK}},
  author={Badenbroek, Riley and Dahl, Joachim},
  journal={Optim. Methods Softw.},
  volume={37},
  number={3},
  pages={1027--1064},
  year={2022},
  publisher={Taylor \& Francis}
}

@article{skajaa2015homogeneous,
  title={A homogeneous interior-point algorithm for nonsymmetric convex conic optimization},
  author={Skajaa, Anders and Ye, Yinyu},
  journal={Math. Program.},
  volume={150},
  number={2},
  pages={391--422},
  year={2015},
  publisher={Springer}
}

@article{papp2017homogeneous,
  title={On ``A Homogeneous Interior-Point Algorithm for Non-Symmetric Convex Conic Optimization"},
  author={Papp, D{\'a}vid and Y{\i}ld{\i}z, Sercan},
  journal={arXiv preprint arXiv:1712.00492},
  year={2017}
}

@article{dahl2022primal,
  title={A primal-dual interior-point algorithm for nonsymmetric exponential-cone optimization},
  author={Dahl, Joachim and Andersen, Erling D},
  journal={Math. Program.},
  volume={194},
  number={1},
  pages={341--370},
  year={2022},
  publisher={Springer}
}

@book{renegar2001mathematical,
  title={A Mathematical View of Interior-Point Methods in Convex Optimization},
  author={Renegar, James},
  year={2001},
  publisher={SIAM},
  address = {Philadelphia}
}

@article{papp2022alfonso,
  title={Alfonso: \text{MATLAB} package for nonsymmetric conic optimization},
  author={Papp, D{\'a}vid and Y{\i}ld{\i}z, Sercan},
  journal={INFORMS J. Comput.},
  volume={34},
  number={1},
  pages={11--19},
  year={2022},
  publisher={INFORMS}
}

@article{he2024qics,
  title={\text{QICS}: Quantum information conic solver},
  author={He, Kerry and Saunderson, James and Fawzi, Hamza},
  journal={arXiv preprint arXiv:2410.17803},
  year={2024}
}

@article{karimi2024domain,
  title={Domain-\text{D}riven \text{S}olver (\text{DDS}) Version 2.1: a \text{MATLAB}-based software package for convex optimization problems in domain-driven form},
  author={Karimi, Mehdi and Tun{\c{c}}el, Levent},
  journal={Math. Program. Comput.},
  volume={16},
  number={1},
  pages={37--92},
  year={2024},
  publisher={Springer}
}

@article{coey2022solving,
  title={Solving natural conic formulations with \text{H}ypatia.jl},
  author={Coey, Chris and Kapelevich, Lea and Vielma, Juan Pablo},
  journal={INFORMS J. Comput.},
  volume={34},
  number={5},
  pages={2686--2699},
  year={2022},
  publisher={INFORMS}
}

@article{kapelevich2024computing,
  title={Computing conjugate barrier information for nonsymmetric cones},
  author={Kapelevich, Lea and Andersen, Erling D and Vielma, Juan Pablo},
  journal={J. Optim. Theory Appl.},
  volume={202},
  number={1},
  pages={271--295},
  year={2024},
  publisher={Springer}
}

@article{dolan2002benchmarking,
  author  = {Elizabeth D. Dolan and Jorge J. Mor{\'e}},
  title   = {Benchmarking optimization software with performance profiles},
  journal = {Math. Program.},
  volume  = {91},
  number  = {2},
  pages   = {201--213},
  year    = {2002},
  //doi     = {10.1007/s101070100263}
}

@article{nesterov1998primal,
  title={Primal-dual interior-point methods for self-scaled cones},
  author={Nesterov, Yu E and Todd, Michael J},
  journal={SIAM J. Optim.},
  volume={8},
  number={2},
  pages={324--364},
  year={1998},
  publisher={SIAM}
}

@article{papp2025interior,
  title={Interior-point algorithms with full \text{N}ewton steps for nonsymmetric convex conic optimization},
  author={Papp, D{\'a}vid and Varga, Anita},
  journal={SIAM J. Optim.},
  volume={35},
  number={4},
  pages={2490--2517},
  year={2025},
  publisher={SIAM}
}

@article{nesterov2012towards,
  title={Towards non-symmetric conic optimization},
  author={Nesterov, Yurii},
  journal={Optim. Methods Softw.},
  volume={27},
  number={4-5},
  pages={893--917},
  year={2012},
  publisher={Taylor \& Francis}
}

@article{tunccel2001generalization,
  title={Generalization of primal-dual interior-point methods to convex optimization problems in conic form},
  author={Tun{\c{c}}el, Levent},
  journal={Found. Comput. Math.},
  volume={1},
  number={3},
  pages={229--254},
  year={2001},
  publisher={Springer}
}

@article{chen2025efficient,
  author  = {Chen, Y. and Goulart, P.},
  title   = {An Efficient Implementation of Interior-Point Methods for a Class of Nonsymmetric Cones},
  journal = {J. Optim. Theory Appl.},
  volume  = {204},
  number  = {33},
  year    = {2025},
  publisher={Springer}
}

@article{coey2022performance,
    title={Performance enhancements for a generic conic interior point algorithm},
    author={Chris Coey and Lea Kapelevich and Juan Pablo Vielma},
    year={2023},
    journal={Math. Program. Comput.},
    publisher={Springer},
    volume={15},
    pages={53--101}
}

@article{fawzi2023optimal,
  title={Optimal self-concordant barriers for quantum relative entropies},
  author={Fawzi, Hamza and Saunderson, James},
  journal={SIAM J. Optim.},
  volume={33},
  number={4},
  pages={2858--2884},
  year={2023},
  publisher={SIAM}
}

@article{he2026operator,
  title={Operator convexity along lines, self-concordance, and sandwiched \text{R}{\'e}nyi entropies},
  author={He, Kerry and Saunderson, James and Fawzi, Hamza},
  journal={Math. Program.},
  pages={1--30},
  year={2026},
  publisher={Springer}
}

@book{nesterov1994interior,
  title={Interior-Point Polynomial Algorithms in Convex Programming},
  author={Nesterov, Yurii and Nemirovskii, Arkadii},
  year={1994},
  publisher={SIAM},
  address = {Philadelphia}
}

@article{nesterov1997self,
  title={Self-scaled barriers and interior-point methods for convex programming},
  author={Nesterov, Yu E and Todd, Michael J},
  journal={Math. Oper. Res.},
  volume={22},
  number={1},
  pages={1--42},
  year={1997},
  publisher={INFORMS}
}

@article{myklebust2014interior,
  title={Interior-point algorithms for convex optimization based on primal-dual metrics},
  author={Myklebust, Tor and Tun{\c{c}}el, Levent},
  journal={arXiv preprint arXiv:1411.2129},
  year={2014}
}

@book{wright1997primal,
  title={Primal-Dual Interior-Point Methods},
  author={Wright, Stephen J},
  year={1997},
  publisher={SIAM}
}

@article{chandrasekaran2017relative,
  title={Relative entropy optimization and its applications},
  author={Chandrasekaran, Venkat and Shah, Parikshit},
  journal={Math. Program.},
  volume={161},
  number={1},
  pages={1--32},
  year={2017},
  publisher={Springer}
}

@article{chandrasekaran2016relative,
  title={Relative entropy relaxations for signomial optimization},
  author={Chandrasekaran, Venkat and Shah, Parikshit},
  journal={SIAM J. Optim.},
  volume={26},
  number={2},
  pages={1147--1173},
  year={2016},
  publisher={SIAM}
}

@article{karimi2025efficient,
  title={Efficient implementation of interior-point methods for quantum relative entropy},
  author={Karimi, Mehdi and Tun{\c{c}}el, Levent},
  journal={INFORMS J. Comput.},
  volume={37},
  number={1},
  pages={3--21},
  year={2025},
  publisher={INFORMS}
}

@article{lin2024generalized,
  title={Generalized power cones: optimal error bounds and automorphisms},
  author={Lin, Ying and Lindstrom, Scott B and Louren{\c{c}}o, Bruno F and Pong, Ting Kei},
  journal={SIAM J. Optim.},
  volume={34},
  number={2},
  pages={1316--1340},
  year={2024},
  publisher={SIAM}
}

@article{lindstrom2023error,
  title={Error bounds, facial residual functions and applications to the exponential cone},
  author={Lindstrom, Scott B and Louren{\c{c}}o, Bruno F and Pong, Ting Kei},
  journal={Math. Program.},
  volume={200},
  number={1},
  pages={229--278},
  year={2023},
  publisher={Springer}
}

@book{golub2013matrix,
  title     = {Matrix Computations},
  author    = {Golub, Gene H. and Van Loan, Charles F.},
  edition   = {4th},
  publisher = {Johns Hopkins University Press},
  address   = {Baltimore},
  year      = {2013}
}

@article{hauser2002self,
  title={Self-scaled barrier functions on symmetric cones and their classification},
  author={Hauser, Raphael A and G{\"u}ler, Osman},
  journal={Found. Comput. Math.},
  volume={2},
  number={2},
  pages={121--143},
  year={2002},
  publisher={Springer}
}

@article{tunccel1998primal,
  title={Primal-dual symmetry and scale invariance of interior-point algorithms for convex optimization},
  author={Tun{\c{c}}el, Levent},
  journal={Math. Oper. Res.},
  volume={23},
  number={3},
  pages={708--718},
  year={1998},
  publisher={INFORMS}
}

@article{roy2022self,
  title={On self-concordant barriers for generalized power cones},
  author={Roy, Scott and Xiao, Lin},
  journal={Optim. Lett.},
  volume={16},
  number={2},
  pages={681--694},
  year={2022},
  publisher={Springer}
}

@article{goulart2026clarabel,
  title={Clarabel: An interior-point solver for conic programs with quadratic objectives},
  author={Goulart, Paul J and Chen, Yuwen},
  journal={Math. Program. Comput.},
  pages={1--83},
  year={2026},
  publisher={Springer}
}
\end{document}